\documentclass{scrartcl}
\usepackage[utf8]{inputenc}
\usepackage{amsfonts,amssymb, amsmath, amsthm}
\usepackage{mathrsfs}
\usepackage{stmaryrd} %
\usepackage{xcolor}
\usepackage{xparse}
\usepackage{enumitem}

\usepackage{tikz}
\usepackage{tikz-cd}
\usetikzlibrary{cd}

\usepackage{hyperref}

\hypersetup{     %
    colorlinks   = true,
    urlcolor     = green!60!black,
    linkcolor    = green!60!black,
    citecolor    = green!60!black,
    linktocpage  = true,
    breaklinks   = true,
    unicode
}

\counterwithin{paragraph}{subsection}
\counterwithin{equation}{subsection}

\theoremstyle{plain}
\newtheorem{lemma}{Lemma}[section]
\newtheorem{proposition}[lemma]{Proposition}
\newtheorem{corollary}[lemma]{Corollary}

\newtheorem{theorem}[lemma]{Theorem}

\theoremstyle{definition}
\newtheorem{remark}[lemma]{Remark}

\newtheorem{example}[lemma]{Example}

\newenvironment{axioms}
    {\begin{enumerate}[labelindent=3em, leftmargin=*, itemsep=0.8em, topsep=1em]}
    {\end{enumerate}}

\makeatletter
\def\namedlabel#1#2{\begingroup%
  #2%
  \def\@currentlabel{#2}%
  \phantomsection\label{#1}\endgroup
}
\makeatother

\newcommand\Ps{\raisebox{.17\baselineskip}{\Large\ensuremath{\wp}}} %
\newcommand{\inv}{^{-1}}
\newcommand\sue{\subseteq}
\newcommand{\op}{^{op}}
\newcommand{\sm}{{\setminus}}  %

\newcommand\two{\mathbf{2}}

\newcommand\mee{\wedge}
\newcommand\bigmee{\bigwedge}
\newcommand\bwe{\bigwedge}
\newcommand\bve{\bigvee}
\newcommand\bca{\bigcap}

\NewDocumentCommand{\bigfitvee}{ O{L} O{} }{{\textstyle\bigvee^{\So(#1)}_{#2}\,}}

\newcommand\upset  {\ensuremath{\mathord{\uparrow}\mkern1mu}}

\newcommand{\Up}{\mathsf{Up}}
\newcommand{\Idl}{\mathsf{Idl}}
\newcommand{\Fi}{\mathsf{Filt}}
\newcommand\sqleq{\sqsubseteq}

\newcommand{\clOp}{\mathsf{cl}}
\newcommand{\inOp}{\mathsf{int}}
\newcommand{\J}{\mathcal{J}}    %
\newcommand{\M}{\mathcal{M}}    %
\newcommand{\fix}{\mathsf{fix}}

\newcommand\diff{\mathbin{\raisebox{0.08em}{$\smallsetminus$}}}

\newcommand\GC{\ensuremath{\mathcal{G}}}

\newcommand\xe[1][X]{i_{#1}}  %
\newcommand\ye[1][Y]{i_{#1}}  %

\newcommand\femb[1][]{\kappa_{#1}}
\newcommand\fembF{\femb[\F]}
\newcommand\lemb[1][]{e_{#1}}
\newcommand\lembF{\lemb[\F]}

\newcommand\ce{^\delta} %
\newcommand\fe{^\F}     %

\newcommand{\F}{\mathcal{F}}    %
\newcommand{\Int}{\J}  %
\newcommand{\FJ}{\mathcal{FJ}}  %

\newcommand{\pt}{\mathsf{pt}}

\newcommand{\B}{\mathfrak{B}}   %

\newcommand{\bs}{\mathfrak{b}}
\newcommand{\os}{\mathfrak{o}}  %
\newcommand{\cs}{\mathfrak{c}}  %
  \newcommand{\of}{\mathfrak{of}}  %
  \newcommand{\cf}{\mathfrak{cf}}  %
\newcommand{\spa}{\mathsf{sp}}
\newcommand\Sl{\mathcal{S}}
\newcommand\SL{\Sl(L)}
    \newcommand\Emp{\mathsf{O}}
    \newcommand\Fll{L}

\newcommand\So{\Sl_o}
    \newcommand\Sb{\Sl_b}
    \newcommand\Sc{\Sl_c}
    \newcommand\Sk{\Sl_k}
        \newcommand\Ssp{\Sl_{sp}}

\newcommand\Slc{\Sl_{lc}}
\newcommand\Sco{\Sl_{co}}
\newcommand\Sop{\Sl_{*}}

\newcommand\fit{\mathtt{fit}}  %

\newcommand\SE{\mathsf{SE}}          %
    \newcommand\Ex{\mathsf{E}}      %
        \newcommand\Cl{\mathsf{C}}  %
        \newcommand\R{\mathsf{R}}    %
    \newcommand\SO{\mathsf{SO}}      %
        \newcommand\CP{\mathsf{CP}}  %

\newcommand\LCl{{\mathsf{L}\Cl}}  %
\newcommand\Pri{\mathsf{Pri}}  %

\newcommand\spp{^\#}  %
\newcommand\psc{^*}   %

\newcommand\con{\mathsf{con}}
\newcommand\tot{\mathsf{tot}}

\newcommand\stf{\mathsf{fi}}  %
\newcommand\fts{\mathsf{su}}  %
\newcommand\ftsC{J}  %

\newcommand\ARG{\,\text{-}\,}

\newcommand\df[1]{\emph{#1}}

\newcommand\ee[1]{\enspace #1 \enspace}
\newcommand\ete[1]{\enspace\text{#1}\enspace}
\newcommand\qtq[1]{\quad\text{#1}\quad}

\newcommand\TOCITE[1][]{{\color{blue}[00]}}

\begin{document}

\title{Canonical extensions via fitted sublocales}
\author{Tomáš Jakl \and Anna Laura Suarez}
\maketitle

\begin{abstract}
    \emph{Abstract.}
     We build on a result stating that the frame $\SE(L)$ of \emph{strongly exact} filters for a frame $L$ is anti-isomorphic to the coframe $\So(L)$ of fitted sublocales. The collection $\Ex(L)$ of \emph{exact} filters of $L$ is known to be a sublocale of this frame. We consider several other subcollections of $\SE(L)$: the collections $\J(\CP(L))$ and $\J(\SO(L))$ of intersections of completely prime and Scott-open filters, respectively, and the collection $\R(L)$ of regular elements of the frame of filters. We show that all of these are sublocales of $\SE(L)$, and as such they correspond to subcolocales of $\So(L)$, which all turn out to have a concise description. By using the theory of polarities of Birkhoff, one can show that all of the structures mentioned above enjoy universal properties which are variations of that of the canonical extension. We also show how some of these subcollections can be described as polarities and give three new equivalent definitions of subfitness in terms of the lattice of filters.
\end{abstract}

\par\noindent\rule{\textwidth}{0.4pt}
\vspace{-2.3em}
\tableofcontents
\par\noindent\rule{\textwidth}{0.4pt}

\section{Introduction}

The paper \cite{jakl20} made the first exploratory steps in canonical extensions of frames. Since then there has been a lot of development of different types of frame extensions which, crucially, were based on complete lattices of sublocales build from closed \cite{picado2019Sc}, open \cite{moshier2020exact} or complemented \cite{arrieta2021complemented} sublocales.
In this paper we revisit the approach of \cite{jakl20} in order to provide an organised manner of explaining and extending these recent results from the point of view of canonical extensions.

Canonical extension is a type of completion of ordered structures, such as Boolean algebras, distributive lattices, and posets which retains many properties of the original structure. Canonical extensions were originally constructed by Jonss\'on and Tarski for Boolean algebras \cite{jonssontarski1951boolean,jonssontarski1952boolean}, to give a topological semantics to modal logic. Over the years their universal properties were identified and the construction was freed of the use of Stone duality~\cite{gehrke2001bounded,gehrke2004bounded}.
Canonical extensions are invaluable when extending Stone duality (and its variant) to algebras with additional operators, such as modalities.
For an overview, we refer the reader to \cite{gehrke2018canonical} and to the upcoming book \cite{GFbook}. In \cite{jakl20}, canonical extensions are studied for locally compact frames, thus extending the construction to the setting of pointfree topology. An application of the theory of \emph{polarities} by Birkhoff \cite{birkhoff48} to frames enables us to systematically construct, for any frame $L$, a concrete structure which enjoys a variation of the properties \emph{density} and \emph{compactness} characterizing the canonical extension construction. This structures is realized as the collection of intersections of Scott-open filters of $L$.

On the other hand, in frame theory, the collection \(\Sl(L)\) of sublocales (i.e.\ pointfree subspaces) has been viewed for many years as a suitable discretisation of the original frame \(L\). However, recently a few sublattices of \(\Sl(L)\) have been studied that often have many desirable properties that \(\Sl(L)\) fails to have. Probably the most prominent examples are \(\Sc(L)\) and \(\So(L)\), the lattices of joins of closed sublocales and fitted sublocales, respectively, see e.g.~\cite{picado2017boolean}, \cite{moshier2020exact}.

\medskip

In this paper, we show that these and other collections of sublocales are all variations of the canonical extension construction. Our starting point is a result in \cite{moshier2020exact} which provides a useful link between fitted sublocales and filters. There, it is shown that \(\So(L)\) is isomorphic to the opposite of the frame \(\SE(L)\) of the so-called strongly exact filters. In this paper, we add to the picture the following collections.
\begin{itemize}
    \item The collection \(\Ex(L)\) of \emph{exact} filters;
    \item The collection $\R(L)$ of \emph{regular} filters, i.e.\ the regular elements of the frame $\Fi(L)$;
    \item The collection $\CP(L)$ of completely prime filters of $L$;
    \item The collection $\SO(L)$ of Scott-open filters of $L$.
\end{itemize}

We come to proving that we have for every frame $L$ the following poset of sublocale inclusions, where $\J$ denotes the closure under arbitrary intersections.

\begin{equation}
\begin{tikzcd}[row sep=1em, column sep=1.5em]
    \R(L) \rar[hook] & \Ex(L) \ar[hook]{rd} \\
        && \SE(L) \rar[hook] & \Fi(L) \\
    \J(\CP(L)) \rar[hook] & \J(\SO(L)) \ar[hook]{ru}
\end{tikzcd}
\label{eq:fi-inc}
\end{equation}

We use the correspondence from \cite{moshier2020exact} and identify the distinguished classes of fitted sublocales that these classes of filters correspond to. It turns out that all of these structures are variations of the canonical extension construction, as explained above. Particularly nice is the case where $L$ is a \emph{fit} frame; for such frames the poset above gives the poset of subcolocale inclusions

\begin{equation}
\begin{tikzcd}[row sep=1em, column sep=1.5em]
 & \mathcal{S}_{c}(L) =\mathcal{S}_b(L)\ar[hook]{rd} \\
        && \mathcal{S}_{o}(L)  \\
    \mathcal{S}_{sp}(L) \rar[hook] & \mathcal{S}_{k}(L)  \ar[hook]{ru}
\end{tikzcd}
\label{eq:sl-inc}
\end{equation}
where $\mathcal{S}_b(L)$ is the Booleanization of \(\Sl(L)\), studied in \cite{arrieta2021complemented}, $\mathcal{S}_{sp}(L)$ is the collection of spatial sublocales, and $\mathcal{S}_{k}(L)$ that of joins of compact sublocales.

Lastly, we show that the theory of polarities often gives very concise abstract descriptions of the structures above, which in some cases does not mention the notion of sublocale.
Thereby we arrive at three new equivalent definitions of subfitness in terms of filters, see Propositions~\ref{p:sfre}, \ref{p:subfit-joins-closed}, and \ref{p:open-closed-surprises}, describe each of the given classes of sublocales uniformly by a simple universal property and, also, characterise preservation of distinguished meets for embeddings into the extensions.

\section{Frame theory preliminaries}

Our main reference is the book \cite{pp2012book} (or the briefer and more recent \cite{pp2021notes}). We will be strictly using the following notation throughout the paper. For a function \(f\colon X\to Y\), we write \(f[A] = \{ f(a) \mid a\in A\}\) for the forward image of the subset \(A\sue X\) and, similarly, we write \(f\inv[B] = \{ x\in X \mid f(x) \in B\}\) for the preimage of \(B\sue Y\).

Further, by a \df{filter} \(F\) of a lattice \(L\) we mean a non-empty upwards closed subset \(F \sue L\) which is closed under finite meets. We say that \(F\) is \df{proper} when \(F \not= L\) and \df{principal} if \(F = \upset a = \{ x \mid a \leq x \}\) for some \(a\).

\subsection{Frames}
\label{s:frm-basics}

A \df{frame} is a complete lattice \(L\) satisfying the following distributivity law.
\[
    (\forall A\sue L, b\in L)
    \qquad
    (\bigvee A) \mee b
    = \bigvee \{ a \mee b \mid a \in A \}
\]
This law implies that every frame is also a Heyting algebra, with the Heyting implication obtained as the right adjoint to meets, i.e.\ \(a \mee b \leq c\) iff \(a \leq b \to c\) for any \(a,b,c\in L\).
However, \df{frame homomorphisms} are only required to preserve infinite suprema and finite infima, including \(0\) and \(1\), respectively. Heyting implication is typically not preserved.

\df{Points} of a frame \(L\) can be defined in three equivalent ways. They are given as
\begin{itemize}
    \item frame homomorphisms \(L \to \two\), where \(\two = (0 < 1)\) is the two-element frame,
    \item completely prime filters \(P \sue L\), that is, proper filters on \(L\) which satisfy, for any \(A\sue L\), that \(\bigvee A \in P\) implies that \(a\in P\) for some \(a\in A\), or
    \item prime elements \(p \in L\), that is, elements such that \(x \mee y \leq p\) implies \(x\leq p\) or \(y\leq p\).
\end{itemize}
Every completely prime filter \(P\) uniquely determines a frame homomorphism \(L \to \two\) as the characteristic function of \(P\) and, on the other hand, the prime element \(p\) corresponding to \(P\) is obtained by the join \(\bigvee \{a\in L \mid a\notin P\}\).

The set of points \(\pt(L)\) of \(L\) can be endowed with a topology consisting of opens of the form \(\{ P \mid a \in P\}\) for some \(a\in L\). Conversely, given a topological space \(X = (X,\tau)\), its lattice of opens \(\Omega(X) = \tau\) ordered by set inclusion is a frame.
Frames for which \(\Omega(\pt(L)) \cong L\) are called \df{spatial}. Conversely, spaces for which \(\pt(\Omega(X))\) is homeomorphic to \(X\) are called \df{sober}.

\subsection{Locales and sublocales}

The category of frames and frame homomorphisms is dual to the category of topological spaces in the sense that the operations \(\Omega(\ARG)\) and \(\pt(\ARG)\) extend to contravariant adjoint functors.
The geometric point of view is retained if, instead, we work with the category of \df{locales}. Objects of this category are still frames (usually referred to as locales) and morphisms are \df{localic maps}, that is, maps \(f\colon L \to M\) such that their (unique) left adjoint is a frame homomorphism \(h\colon M \to L\): we have \(h(a) \leq b\) iff \(a \leq f(b)\) for \(a\in M, b\in L\).

The generalisation of the notion of a subspace of a space is that of a \df{sublocale}. Similarly to subspaces, sublocales correspond to images of injective localic maps. Equivalently, these are subsets \(S\sue L\) such that
\begin{axioms}
    \item[(S1)] \(S\) is closed under all meets in \(L\), and
    \item[(S2)] for every \(s\in S\) and \(a\in L\), \(a \to s \in S\).
\end{axioms}
The collection \(\Sl(L)\) of all sublocales of \(L\) is a complete lattice, with meets \(\bigwedge_i S_i\) given by intersections \(\bigcap_i S_i\) and joins \(\bigvee_i S_i\) given by \(\{\bigmee A \mid A \sue \bigcup_i S_i\}\).

Every \(a\in L\), induces the \df{open sublocale} \(\os(a)\) corresponding to \(a\) and, its complement in~\(\Sl(L)\), the \df{closed sublocale} \(\cs(a)\). Being defined by
\[
    \os(a) = \{ a \to b \mid b\ \in L\} = \{ x \in L \mid a \to x = x \}
    \qtq{and}
    \cs(a) = \upset a,
\]
they exhibit a lot of expected properties of open and closed subspaces, for example:
\begin{align*}
    \os(1) = \Fll
    &\qquad
    \os(0) = \Emp
    &
    \textstyle
    \bigvee_i \os(a_i) = \os(\bigvee_i a_i)
    &\qquad
    \os(a) \cap \os(b) = \os(a \mee b)
    \\
    \cs(1) = \Emp
    &\qquad
    \cs(0) = \Fll
    &
    \textstyle
    \bigcap_i \cs(a_i) = \cs(\bigvee_i a_i)
    &\qquad
    \cs(a) \vee \cs(b) = \cs(a \mee b)
\end{align*}
Here \(\Emp = \{1\}\) and
\(\Fll\)
represent the smallest and largest sublocales of \(L\), respectively. This means that \(\cs(a) \cap \os(a) = \Emp\) and \(\cs(a) \vee \os(a) = \Fll\).

\subsection{Coframes, subcolocales, and supplements}

We also make use of notions dual to those introduced above. For example, we say that a poset \(C\) is a \df{coframe} (or \df{colocale}) if the same poset \(C\op\) but in the opposite order is a frame. Similarly, \(S \sue C\) is a \df{subcolocale} if \(S\op\) is a sublocale of \(C\op\).

The coHeyting implication in a coframe \(C\), called \df{difference} and denoted \(a \diff b\), satisfies
\(a \diff b \leq c\) iff \(a \leq b \vee c\).
Then, similarly to the \df{pseudocomplement}
\[
    a^* = a \to 0
\]
of an \(a\) in a frame, the dual notion is called \df{supplement} of an \(a\) in a coframe and is calculated as
\[a\spp = 1 \diff a.\]
An important example of a coframe is \(\Sl(L)\).

\subsection{Filters of frames}
\label{s:filt}

In this text we aim to prioritise the geometric point of view on filters and we think of a filter \(F\) of \(L\) as a representative for the sublocale corresponding to the intersection
\begin{align}
    \fts(F) =
    \bigcap_{a\in F} \os(a)
    \label{eq:fi-to-sl}
\end{align}
in \(\Sl(L)\). Therefore, unless said otherwise, we always order the collection of all filters \(\Fi(L)\) of a frame \(L\) by the reverse inclusion order. To emphasise this, we denote the filter order by \(\sqleq\). Under this order the correspondence from \eqref{eq:fi-to-sl} gives a monotone mapping from filters to sublocales:
\[
    F \sqleq G \ee\implies \fts(F) \sue \fts(G)
\]
Then, \(\Fi(L)\) is a complete lattice with joins and binary meets computed by:
\[
    \bigsqcup A = \bigcap A
    \qtq{and}
    F \sqcap G = \{ f \mee g \mid f\in F\, g\in G\}
\]
In fact, \(\Fi(L)\) is a coframe, with the \df{difference} given as follows.
\[
    H \diff G = \{ a\in L \mid \forall b\in G.\ b \vee a \in H\}
\]
Indeed, \(H \diff G \sqleq F\) iff \(H \sqleq F \sqcup G\), see e.g.\ Section~5.1 in \cite{moshier2020exact}.

\paragraph{Distinguished classes filters.}
\label{par:filt-cls}
In this paper we are looking at various classical classes of filters on a frame. Namely, we consider
\begin{itemize}
    \item \df{completely prime filters} \(\CP(L)\), introduced already in Section~\ref{s:frm-basics},
    \item \df{Scott-open filters} \(\SO(L)\), that is, filters \(F\) such that for any directed\footnote{Recall that a set \(A\) is directed if for any \(a,b\in A\) there is some \(c\in A\) such that \(a\leq c\) and \(b\leq c\).} \(A\sue L\), if \(\bigvee A \in F\) then \(a\in F\) for some \(a\in A\),
    \item \df{exact filters} \(\Ex(L)\), that is, filters \(F\) closed under \df{exact meets} which are the meets \(\bigmee M\) for some \(M\sue L\) such that, for any \(b\), \((\bigmee M) \vee b = \bigmee_{a\in M} (a \vee b)\), and
    \item \df{strongly exact filters} \(\SE(L)\), that is, filters \(F\) closed under \df{strongly exact meets}, which are the meets \(\bigmee M\) for some \(M\sue L\) such that \(\bigcap_{a\in M} \os(a) = \os(\bigmee M)\).
\end{itemize}

We view all classes of filters as subposets of \(\Fi(L)\),  in the \(\sqleq\) order.
Crucially, Theorem~4.5 in \cite{moshier2020exact} characterises strongly exact filters also in the following way.
\begin{lemma}
    A filter \(F\sue L\) is strongly exact if and only if, for any \(b\in L\),
    \begin{equation}
        \bigcap_{a\in F} \os(a) \sue \os(b)
        \implies
        b\in F
        \label{eq:se-fi}
    \end{equation}
    \label{l:SE-char}
\end{lemma}

\section{Canonical extensions and polarities}

\subsection{Canonical extensions}

The motivation for our work comes from canonical extensions of (bounded) distributive lattices. Given such lattice \(D\), its canonical extension is an embedding \(e\colon D \hookrightarrow D\ce\) into a complete lattice \(D\ce\) such that:
\begin{axioms}
    \item[(D)] Every element of \(D\ce\) is a join of meets and a meet of joins of elements in the image~\(e[D]\).
    \item[(C)] If $\bigmee e[F] \leq \bigvee e[I]$ for some filter $F\sue D$ and ideal $I\sue D$, then $F \cap I \not= \emptyset$.
\end{axioms}

The existence, unicity and universal properties of \(D \hookrightarrow D\ce\) follows from the theory of polarities which we discuss in Section~\ref{s:polarities} below.

\paragraph{Canonical extensions of frames.}
\label{par:can-frm}
The elements of the form \(\bigmee e[F]\) and \(\bigvee e[I]\), for a filter \(F\sue D\) and ideal \(I\sue D\), play an important role in the definition above. In fact, axiom (D) equivalently says that every element of \(D\ce\) is both a join of elements of the form \(\bigmee e[F]\) and a meet elements of the form~\(\bigvee e[I]\).

Recall, from Stone duality for distributive lattices, that ideals correspond to opens of the spectral space \(X\) dual to \(D\), i.e.\ \(\Idl(D) \cong \Omega(X)\) \cite{priestley1984ordered}. Similarly, filters correspond to compact saturated subsets of \(X\). However, by the Hofmann--Mislove Theorem~\cite{hofmann2006local}, compact saturated subsets are in a correspondence with Scott-open filters on the frame~\(\Omega(X)\), that is, filters \(F\) such that for any directed family \(\{a_i\}_i\) such that \(\bigvee_i^\uparrow a_i \in F\) there is some \(i\) such that \(a_i \in F\).

This inspired the frame-theoretic definition of canonical extensions of frames in \cite{jakl20}. Given a frame \(L\), we define its canonical extension to be a monotone mapping \(e\colon L \to L^\SO\) such that:
\begin{axioms}
    \item[\namedlabel{ax:dso}{(D\({}^\SO\))}] Every element of $L^\SO$ is a join of elements of the form $\bigmee e[F]$ for some Scott-open filter \(F\sue L\) and a meet of elements of the form \(e(a)\) for some \(a\in L\).

    \item[\namedlabel{ax:cso}{(C\({}^\SO\))}] If $\bigmee e[F] \leq e(a)$ for some Scott-open filter $F \sue L$ and $a\in L$, then $a\in F$.
\end{axioms}

Observe that by the above discussion we have that \(D\ce \cong \Idl(D)^\SO\) and so canonical extensions of frames generalises canonical extensions for distributive lattices. Similarly to distributive lattices, the canonical extension \(L \to L^\SO\) always exists and is unique up to isomorphism. In case when \(L = \Omega(X)\) for some space \(X\), \(e\) is simply the embedding \(\Omega(X) \hookrightarrow \Up(X)\) where \(\Up(X)\) is the poset of saturated subsets, that is, upsets in the specialization order (cf.\ Example 3.5 in \cite{jakl20}).

\begin{remark}
    In the above discussion we omitted an important part of the theory which concerns with extensions of (not necessarily monotone) maps between distributive lattices \(D \to E\) to maps between their canonical extensions \(D\ce \to E\ce\). The theory of these extensions is what makes canonical extensions useful for semantics. Similar extensions theorems can be proved for canonical extensions of frames as well, cf.\ \cite{jakl20}.
\end{remark}

\subsection{Polarities}
\label{s:polarities}

The theory of polarities goes back all the way to Birkhoff~\cite[Chapter V]{birkhoff48}. It can be used for presenting different types of completions of posets known from the literature. We use polarities for generalisations of canonical extensions, and showing their existence, uniqueness and universal properties.
In this section we mostly follow the description from~\cite{gehrke06}.

\paragraph{}
\label{p:polarity}
A \df{polarity} is a tuple \(P=(X,Y,Z)\) where $Z \sue X\times Y$ is a relation between sets $X$ and $Y$. This data induces a pair of antitone functions between the powersets on \(X\) and \(Y\) (both taken with the subset orders):
\begin{align*}
    p\colon \Ps(X) \to \Ps(Y), \quad M \mapsto \{ y\in Y ~|~ \forall x\in M.\ x Z y\} \\
    q\colon \Ps(Y) \to \Ps(X), \quad N \mapsto \{ x\in X ~|~ \forall y\in N.\ x Z y\}
\end{align*}
Note that these maps are adjoint to each other, since $N \sue p(M)$ iff $M \sue q(N)$. Therefore, $q \circ p$ is a closure operator\footnote{By a \df{closure operator} we mean a monotone mapping \(\delta \colon P \to P\) on a poset \(P\) such that \(a \leq \delta(a)\) and \(\delta(\delta(a)) = \delta(a)\) for any \(a\in P\).} on the complete Boolean algebra $\Ps(X)$ and, dually, \(p \circ q\) is a closure operator on \(\Ps(Y)\).

Recall that any closure operator \(\delta\colon C \to C\) on a complete lattice \(C\) induces a complete lattice of fixpoints \(\fix(\delta) = \{ a \in C \mid \delta(a) = a\}\), with the order derived from \(C\). Furthermore, joins and meets in \(\fix(\delta)\) are computed in \(C\) as follows.
\begin{align}
    \textstyle
    \bigmee^{\fix(\delta)} A \ee{=} \bigmee^C A
    \qtq{and}
    \bigvee^{\fix(\delta)} A \ee{=} \delta(\bigvee^C A)
    \label{eq:fix-cl}
\end{align}

\paragraph{}
\label{par:FP-def}
For a polarity \(P\), define $\GC(P)$ to be the complete lattice of Galois closed sets of \(q \circ p\), that is, set
\[
    \GC(P) = \fix(q\circ p) = \{ M\in \Ps(X) \mid q(p(M)) = M\}.
\]
By \eqref{eq:fix-cl}, joins and meets in $\GC(P)$ are computed in \(\Ps(X)\) as 
\(\bigvee A = qp(\bigcup A)\) and \(\bigmee A = \bigcap A\), respectively.

The embeddings of singletons, $X \to \Ps(X)$ and $Y \to \Ps(Y)$, induces two maps into the lattice of Galois closed sets:
\[ \xe\colon X\to \GC(P),\ x\mapsto qp(\{x\}),
   \ete{and}
   \ye\colon Y\to \GC(P),\ y\mapsto q(\{y\}).
\]
A fundamental property of polarities is that these maps can be used to give an abstract characterisation of $\GC(P)$.

\begin{theorem}\label{t:polarities}
    Let $P = (X, Y, Z)$ be a polarity. Then the complete lattice $C = \GC(P)$ has the following properties.
    \begin{enumerate}
        \item For any $u \in C$,
        \[
            u = \bigvee \{ \xe(x) ~|~ x\in X,\ \xe(x)\leq u\}
            \ete{and}
            u = \bigmee \{ \ye(y) ~|~ y\in Y,\ u \leq \ye(y)\}.
        \]
        \item For any $x\in X$ and $y\in Y$, $\xe(x) \leq \ye(y)$ iff $x Z y$.
        \item For any $\xe'\colon X \to C'$ and $\ye'\colon Y \to C'$ also satisfying the conditions (1) and (2) there is a unique complete lattice isomorphism $\iota\colon C'\to C$ such that $\iota \circ \xe' = \xe$ and $\iota \circ \ye' = \ye$.
    \end{enumerate}
\end{theorem}
\begin{proof}
    Item 1 is Proposition 2.10 and item 2 is Proposition 2.6 in \cite{gehrke06}. The isomorphism \(\iota\colon C' \to C\) in item 3 is obtained by extending the mapping \(\xe'(x) \mapsto \xe(x)\) by taking suprema. For details see \cite{wille82} and also Section 2 of \cite{gehrke06}, by unwrapping the characterisation of \(\GC(P)\) as a Dedekind-MacNeille completion.
\end{proof}

\begin{example}
    The canonical extension \(D\ce\) of a distributive lattice \(D\) and \(L^\SO\) of a frame \(L\) are given by \(\GC(\Fi(D), \Idl(D), Z)\) and \(\GC(\SO(L), L, Z')\), respectively, where \((F,I)\in Z\) iff \(F\cap I \neq \emptyset\) and \(Z' = {\ni}\) i.e.\ \((F,a)\in Z'\) iff \(a\in F\) \cite{gehrke2001bounded, jakl20}.
\end{example}

\begin{remark}
    \label{r:fix-PY}
    It is a standard fact that the poset of fixpoints on one side of an adjunction is isomorphic to the poset of fixpoints on the other side.
    In our case, $\GC(P)$ is isomorphic to the lattice of subsets of $Y$ which are the fixpoints of $p \circ q$. However, the isomorphism \(\fix(q \circ p) \cong \fix(p \circ q)\) is antitone because the maps \(p\) and \(q\) are antitone. In order to obtain a monotone isomorphism we have to order $\Ps(Y)$ by the inverse subset order.

    This point of view becomes important when we discuss concrete descriptions of polarities on filters of a frame in Section~\ref{s:concrete-desc}.
\end{remark}

Furthermore, we also make use of the following simple consequence of Theorem~\ref{t:polarities} which describes the upside-down version of \(\GC(P)\).

\begin{corollary}
    \label{c:GC-op}
    For sets \(X,Y\) and a relation \(Z\sue X\times Y\), \[\GC(X,Y,Z)\op \ee\cong \GC(Y,X,Z\op)\] where \(Z\op = \{ (y,x) \mid (x,y)\in Z\}\).
\end{corollary}

\paragraph{Polarities of posets.}
\label{par:pos-pol}
In the following, we fix a polarity \(P = (X,Y,Z)\) where $(X,\leq)$ and $(Y,\leq)$ are posets. This situation is motivated by our applications where \(Y\) is a frame and \(X\) is a poset of filters on \(Y\).

The following simple but really useful facts demonstrate the power of the abstract characterisation given by Theorem~\ref{t:polarities}.

\begin{lemma}
    \label{l:xe-ye-basics}
    Let \(\xe\) and \(\ye\) be as above. Then:
    \begin{enumerate}
        \item \(\xe(x) \leq \xe(x')\) \ee{iff,} for every $y\in Y$, $x' Z y$ implies $x Z y$.
        \item $\xe$ is monotone \ee{iff}  $x \leq x'$ and $x' Z y$ implies $x Z y$.
        \item $\xe$ is injective \ee{iff} (for all $y\in Y$, $x' Z y \iff x Z y$) implies $x = x'$.
        \item $\xe$ is order-reflecting \ee{iff} whenever ($\forall y\in Y$, $x' Z y$ implies $x Z y$) then $x\leq x'$.
    \end{enumerate}
    And dually:
    \begin{enumerate}
        \setcounter{enumi}{4}
        \item \(\ye(y) \leq \ye(y')\) \ee{iff,} for every $x\in X$, $x Z y$ implies $x Z y'$.
        \item $\ye$ is monotone \ee{iff} $x Z y$ and $y \leq y'$ implies $x Z y'$.
        \item $\ye$ is injective \ee{iff} (for all $x\in X$, $x Z y \iff x Z y'$)  implies $y = y'$.
        \item $\ye$ is order-reflecting \ee{iff} whenever ($\forall x\in X$, $x Z y$ implies $x Z y'$) then $y\leq y'$.
    \end{enumerate}
\end{lemma}
\begin{proof}
    For item 1 observe that, by Theorem~\ref{t:polarities}.1, $\xe(x) \leq \xe(x')$ iff, for every \(y\in Y\), \(\xe(x') \leq \ye(y)\) implies \(\xe(x)\leq \ye(y)\). The rest follows by Theorem~\ref{t:polarities}.2.

    Items 3 and 4 follow immediately from item 1. For item 2, fix \(x \leq x'\). By item 1, we see that \(\xe(x) \leq \xe(x')\) iff $x' Z y$ implies $x Z y$. The second part is completely analogous.
\end{proof}

\noindent
Observe that if \(\xe\) is both monotone and order-reflective then
\[
    x \leq x' \ete{iff,} \forall y\in Y,\ x' Z y \text{ implies } x Z y
\]
and similarly for \(\ye\).

In our examples, we often have that $X$ is a $\vee$-semilattice with directed infima, for example, when $X$ is the set of Scott-open filters. To accommodate various types of join and meet preservation requirements for \(\xe\) we establish the following general lemma.

\begin{lemma}
    \label{l:xe-pres}
    Let $A$ be a subset of $X$.
    Assuming the supremum $\bigvee A$ exists in $X$, then
    \begin{enumerate}
        \item \(\xe(\bigvee A) = \bigvee \xe(A)\) \ee{iff} (for every $a\in A$, $a Z y$) implies $(\bigvee A) Z y$.
    \end{enumerate}
    Further, assuming the infimum $\bigmee A$ exists in $X$
    \begin{enumerate}
        \setcounter{enumi}{1}
        \item if $\xe$ is monotone and order-reflecting, then $\xe(\bigmee A) = \bigmee \xe(A)$.
    \end{enumerate}
\end{lemma}
\begin{proof}
    Observe that, by Theorem~\ref{t:polarities}.1, $\xe(\bigvee A) \leq \bigvee \xe(A)$ is equivalent to: for a fixed~\(y\in Y\),
    \[ (\forall a\in A,\, \xe(a) \leq \ye(y)) \ete{implies} \xe(\bigvee A) \leq \ye(y).\]
    However, by Theorem~\ref{t:polarities}.2, this is equivalent to the claim.

    For the second part, we want to show that $\bigmee \xe(A) \leq \xe(\bigmee A)$. Assume that $\xe(x) \leq \bigmee \xe(A)$, i.e.\ $\xe(x) \leq \xe(a)$ for all $a\in A$. Since $\xe$ reflects the order, $x \leq a$ for every $a\in A$, i.e.\ $x\leq \bigmee A$. Then, since $\xe$ is monotone, also $\xe(x) \leq \xe(\bigmee A)$. As $x$ was chosen arbitrarily, by Theorem~\ref{t:polarities}.1, we obtain that $\bigmee \xe(A) \leq \xe(\bigmee A)$.
\end{proof}

For future reference, we also mention the dual version of this fact.
\begin{lemma}
    \label{l:ye-pres}
    Let $B$ be a subset of $Y$.
    Assuming the infimum $\bigmee B$ exists in $Y$, then
    \begin{enumerate}
        \item $\ye(\bigmee B) = \bigmee \ye(B)$ \ee{iff} (for every $b\in B$, $x Z b$) implies $x Z (\bigmee B)$.
    \end{enumerate}

    Further, assuming the supremum $\bigvee B$ exists in $Y$
    \begin{enumerate}
        \setcounter{enumi}{1}
        \item if $\ye$ is monotone and order-reflecting, then $\ye(\bigvee B) = \bigvee \ye(B)$.
    \end{enumerate}
\end{lemma}

We usually get monotonicity for free. On the other hand, checking that $\xe$ resp.\ $\ye$ is order-reflective is not always immediate. We often get this for free by the following simple fact about injective semilattice homomorphisms.

\begin{lemma}
    \label{l:slat-inject}
    Injective semilattice homomorphism are order-reflective.
\end{lemma}
\begin{proof}
    Let $f\colon S \to S'$ be an injective $\vee$-semilattice homomorphism.
    Assume $f(s) \leq f(s')$. Then $f(s') = f(s) \vee f(s') = f(s \vee s')$ and by injectivity $s\leq s'$ since $s' = s \vee s'$.
\end{proof}

\paragraph{Computing inside complete lattices.}
\label{par:inside-clat}
We finish this overview of the general theory of polarities by a result that allows us to give concrete description of \(\GC(P)\) by computing it as a certain closure within a complete lattice.
To this end, given a subset \(S\) of a complete lattice \(C\), define the following closure and interior operators:
\begin{align*}
    \clOp_S : &\  C \to C
    &
    \inOp_S : &\ C \to C
    \\
    & c \mapsto \bigmee \{ s \in S \mid c \leq s\}
    &
    & c \mapsto \bigvee \{ s \in S \mid s \leq c\}
\end{align*}
Furthermore, write
\begin{equation}
    \J(S)\sue C
    \qtq{and}
    \M(S)\sue C
    \label{eq:joins-clos}
\end{equation}
for the closure of \(S\) under all joins and meets in \(C\), respectively.

\begin{proposition}
    \label{p:int-cl}
    For \(X,Y\sue C\) where \(C\) is a complete lattice and \(\leq\) is the order of~\(C\),
    \[
        \GC(X,Y,\leq) \cong \inOp_X[\M(Y)]  \cong \clOp_Y[\J(X)]
    \]
    and, consequently, also
    \[
        \GC(X,Y,\leq) \cong \GC(X,\M(Y),\leq) \cong \GC(\J(X),Y,\leq) \cong \GC(\J(X),\M(Y),\leq).
    \]
\end{proposition}
\begin{proof}
    Recall from \S\ref{par:FP-def} that \(\GC(X,Y,\leq)\) can be concretely computed as the image \(q[\Ps(Y)]\) for \(p\) and \(q\) as in \S\ref{p:polarity}, induced by \(Z\) set to \({\leq}\). Therefore, \(q[\Ps(Y)]\) consists of the sets of the form
    \[
        \{ x \in X \mid x \leq \bigmee N\}
        \qquad\ete{for some} N \sue Y.
    \]
    Define \(\alpha\colon q[\Ps(Y)] \to \inOp_X[\M(Y)]\) by sending \(S\) to \(\bigvee S\). Observe that this is a well-defined map. Indeed, for any \(N \sue Y\), we have \(\alpha(q(N)) = \inOp_X(\bigmee N)\) and, for \(N,N'\sue Y\), we have that \(q(N) = q(N')\) if and only if, for any \(x\in X\), \(x \leq \bigmee N\) iff \(x\leq \bigmee N'\).

    Next we show that \(\alpha\) is an isomorphism of posets. Since \(\alpha\) is clearly monotone, it is enough to show that it is onto and order-reflective. For the latter, assume \(\alpha(q(N)) \leq \alpha(q(N')\) for some \(N,N'\sue Y\) and let \(x \in q(N)\). Then, by definition, \(x\leq \alpha(q(N)) \leq \alpha(q(N')) \leq \bigmee N'\), giving us \(x\in q(N')\). Therefore \(q(N) \sue q(N')\). For surjectivity, it is enough to observe that a typical element of \(\inOp_X[\M(Y)]\) is of the form \(\inOp_X(\bigmee N) = \alpha(q(N))\) for some \(N \sue Y\). We have shown that \(\GC(X,Y,\leq)\) is isomorphic to \(\inOp_X[\M(Y)]\).

    From this it follows that \(\GC(X,Y,\leq) \cong \GC(X,\M(Y),\leq)\).
    By Corollary~\ref{c:GC-op} we also obtain the dual statement as \(\GC(X,Y,\leq)\op \cong \GC(Y,X,\geq) \cong \clOp_Y[\J(X)]\), viewed as a sub-poset of~\(C\op\).
\end{proof}

\paragraph{The isomorphism concretely.}
\label{par:iso-cl-int}
Observe that the proof gives us that the formula for the isomorphism is given by the closure operation
\[
    \clOp_Y : \inOp_X[\M(Y)]  \xrightarrow{\ee\cong} \clOp_Y[\J(X)].
\]
Indeed, for \(N = p(q(N))\) i.e.\ \(N = \{ y\in Y \mid \inOp_X(\bigmee N) \leq y\}\) and for the mapping \(\beta\colon p[\Ps(X)] \to \clOp_Y[\J(X)]\), \(S \mapsto \bigmee S\) (the map defined dually to \(\alpha\)), we have that \(\clOp_Y(\alpha(q(N))) = \clOp_Y(\inOp_X(\bigmee N)) = \bigmee N = \beta(N)\).

By a dual argument, the inverse map in the opposite direction is given by the interior operation
\[
    \inOp_X : \clOp_Y[\J(X)] \xrightarrow{\ee\cong} \inOp_X[\M(Y)].
\]

\section{Filter extensions}
\label{s:filter-ext}

Recent frame-theoretic constructions appearing in the literature, e.g.\ in \cite{picado2019Sc,moshier2020exact,jakl20}, are examples of polarities specified by a class of filters \(\F \sue \Fi(L)\) on a frame~\(L\). The induced polarity is of the form \((\F, L, \ni)\), i.e.\ it is the polarity \((\F,L,Z)\) with \(Z\) defined by: \(F Z a \iff a \in F\).
\begin{remark}
    Observe that the definition of \(Z\) aligns with our geometrical point of view. Indeed, with the mapping from \eqref{eq:fi-to-sl}, \(F Z a\) implies \(\fts(F) \sue \os(a)\).
    Furthermore, if \(F\) is strongly exact (cf.\ \S\ref{par:filt-cls}) the converse also holds. In fact, all classes of filters we consider in this text are subclasses of strongly exact filters.
\end{remark}

We define the \df{\(\F\)-extension} \(L\fe\) of \(L\) (or simply the \df{filter extension} if \(\F\) is clear from context) as the lattice of Galois closed sets for the polarity induced by \(\F\), that is,
\[
    L\fe = \GC(\F,L,\ni).
\]
Similarly to \cite{jakl20}, we denote the induced maps \(\ye[L]\) and \(\xe[\F]\) by
\[
    \lemb[\F] : L \to L\fe
    \qtq{and}
    \femb[\F] : \F \hookrightarrow L\fe
\]
respectively. Sometimes we simply write \(\lemb\) and \(\femb\) if \(\F\) is clear from the context.

Then, Theorem~\ref{t:polarities} entails the following universal properties of \(\lemb[\F]\) inspired by \ref{ax:dso} and \ref{ax:cso}.
Namely, items 1 and 2 in Theorem~\ref{t:polarities} imply that $\femb[\F](F) = \bigmee \lemb[\F][F]$ and, consequently, \(\lemb[\F]\) is the unique mapping \(e\colon L \to L\fe\) into a complete lattice \(L\fe\) such that:
\begin{axioms}
    \item[\namedlabel{ax:df}{(D\({}^\F\))}] Every element of $L\fe$ is a join of elements of the form $\bigmee e[F]$ for some \(F\in \F\) and a meet of elements of the form \(e(a)\) for some \(a\in L\).

    \item[\namedlabel{ax:cf}{(C\({}^\F\))}] If $\bigmee e[F] \leq e(a)$ for $F \in \F$ and $a\in L$, then $a\in F$.
\end{axioms}

The fact that \(\fembF\) is an embedding as well as many of its and \(\lembF\)'s desirable properties follow automatically from the construction and our observations in~\S\ref{par:pos-pol}. The following is a generalisation of Propositions 3.6 and 5.1 in~\cite{jakl20}.

\begin{proposition}
    \label{p:fe-basic-props}
    For a class of filters \(\F\) ordered by \(\sqleq\), as indicated in Section~\ref{s:filt}, and for \(\fembF\) defined as above we have that
    \begin{enumerate}
        \item $\fembF$ is monotone, injective and order-reflective, and
        \item $\fembF$ preserves existing joins and meets in  $(\F,\sqleq)$.
    \end{enumerate}
    Also, for \(\lembF\) defined as above we have that
    \begin{enumerate}
        \setcounter{enumi}{2}
        \item $\lembF$ is monotone,
        \item \(\lembF\) preserves 0 and finite meets,
        \item if $\lembF$ is injective then it is a frame embedding, i.e.\ it is order-reflective and preserves finite meets as well as arbitrary joins, and
        \item $\lembF$ is injective \ee{iff} $L$ is \df{$\F$-separable}, i.e.\
    \[ a = b \ete{in} L \qtq{iff} \forall F\in \F, \ a\in F \iff b\in F. \]
    \end{enumerate}
\end{proposition}
\begin{proof}
    First, we discuss consequences of Lemma~\ref{l:xe-ye-basics} for the polarity \((\F,L,Z)\) where \(Z = {\ni}\).
    Observe that \(F \sqleq G\), \(G \mathbin{Z} a\) and \(a \leq b\) implies \(F \mathbin{Z} b\), giving us that both \(\fembF\) and \(\lembF\) are monotone. Injectivity of order-reflectivity of \(\fembF\) follows from the fact that \(\sqleq\) and equality of filters is determined by the \(Z\) relation. We have checked (1) and (3). Observe that (6) also follows from the definition of \(Z\). Similarly, (2) is a consequence of Lemma~\ref{l:xe-pres}, (1) and the definition of~\(Z\).

    To check (4), we use Theorem~\ref{t:polarities}. Assume \(\femb(F) \leq \lemb(0)\) for some \(F\in \F\). By \ref{ax:cf}, \(0\in F\), i.e.\ \(F = \upset 0\) is the least filter in the \(\sqleq\) order.
    Since every element \(L\fe\) is a join of elements of the form \(\femb(F)\), we have \(\lemb(0) = 0\).
    Preservation of finite meets is a consequence of Lemma~\ref{l:ye-pres} and the definition of \(Z\) as \(\ni\). Finally, (5) follows from Lemma~\ref{l:slat-inject} and (4).
\end{proof}

\subsection{A concrete description}
\label{s:concrete-desc}
Up until now everything we proved about filter extensions followed from the abstract description given by Theorem~\ref{t:polarities}. In order to make comparison with existing constructions in the literature we need to have an exact description of the lattice \(L\fe\). In Corollary~4.4 of \cite{jakl20} it was shown that \(L^\SO\) can be identified with  the poset of intersections of Scott-open filters. The same is true about general filter extensions. Recall from \eqref{eq:joins-clos} the operation \(\J(\ARG)\) which closes a subset of a complete lattice under all joins.
Then, for \(\F\) viewed as a subset of the complete lattice \(\Fi(L)\), ordered by \(\sqleq\),
\begin{equation}
    \Int(\F) = \{ \bigcap A \mid A \sue \F\}
    \label{eq:int-F}
\end{equation}
is the poset of intersections of filters from \(\F\) since intersections of filters are precisely joins in~\(\Fi(L)\).

\begin{proposition}
    \label{p:fext-concr}
    For any \(\F\), the filter extension \(L\fe\) is isomorphic to \(\Int(\F)\).
\end{proposition}
\begin{proof}
    This follows from Proposition~\ref{p:int-cl}. Observe that, for \(F\in \Fi(L)\) and \(a,b\in L\), we have that \(F Z a\) (i.e.\ \(a\in F\)) iff \(F \sqleq \upset a\) in \(\Fi(L)\) and \(a\leq b\) iff \(\upset a \sqleq \upset b\). Consequently, for \(\F\) and \(\Pri(L) = \{\upset a \mid a\in L\}\) both viewed as subsets of the complete lattice \(\Fi(L)\),
    \[
        L^\F
        \cong \GC(\F,\Pri(L), \sqleq)
        \cong \clOp_{\Pri(L)}[\J(\F)]
        = \J(\F)
    \]
    where the last equality holds because for every \(F\in \Fi(L)\), we have that \(\clOp_{\Pri(L)}(F) = \bigsqcap \{ \upset a \mid a\in F\} = F\).
\end{proof}

Consequently, existing joins in \(L\fe\) are computed as intersections of filters in \(\Fi(L)\). In fact, it is a meet-sublattice of \(\Fi(L)\). The mappings \(\lemb\) and \(\femb\) composed with the isomorphism \(L\fe \cong \Int(\F)\) translate as follows.
\begin{equation}
    \begin{aligned}
        \femb\colon \F &\to \Int(\F)
        &\qquad&&
        \lemb\colon L &\to \Int(\F)
        \\
        F&\mapsto F
        &&&
        a&\mapsto \bigcap \{ F\in \F \mid a\in F\}
    \end{aligned}
    \label{eq:concr-e-k}
\end{equation}

Observe that classes of exact and strongly exact filters are already closed under intersections.
\begin{lemma}
    \(\Int(\SE(L)) = \SE(L)\) and \(\Int(\Ex(L)) = \Ex(L)\).
    \label{l:se-e-j}
    \label{l:SE-intersect}
\end{lemma}
\begin{proof}
    If \(\bigcap_j F_j\) is an intersection of (strongly) exact filters and \(\bwe_i x_i\) is (strongly) exact meet such that \(\{x_i\}_i \sue \bigcap_j F_j\) then also \(\bwe_i x_i \in \bigcap_j F_j\) because each \(F_j\) contains \(\bwe_i x_i\). Consequently, \(\bigcap_j F_j\) is also a (strongly) exact filter.
\end{proof}

\subsection{Specific cases of \texorpdfstring{\(\lemb[\F]\colon L \to L\fe\)}{map e}}

In this subsection we take a look at properties of \(\lemb[\F]\) depending on the class \(\F\). To reduce clutter, instead of writing \(\lemb[\SE(L)]\), \(\femb[\SE(L)]\) and \(L^{\SE(L)}\) we simply write \(\lemb[\SE]\), \(\femb[\SE]\) and \(L^{\SE}\) and similarly for other classes of filters such as \(\Ex(L), \SO(L), \CP(L)\) and those that we introduce later on.

Our starting point is the following fact, where item 3 assumes the concrete description \(L\fe = \Int(\F)\) and \eqref{eq:concr-e-k}.

\begin{proposition}\label{p:generalchar}
For a frame $L$ and a collection $\mathcal{F}$ of its filters, the following are equivalent:
\begin{enumerate}
    \item $\lemb[\F]$ is injective;
    \item The frame $L$ is $\F$-separable;
    \item $\lemb[\F](a)= \upset a$ for all $a\in L$;
    \item The collection $\Int(\F)$ contains all principal filters. 
\end{enumerate}
\end{proposition}
\begin{proof}
    The equivalence of (1) and (2) is precisely item 6 of Proposition~\ref{p:fe-basic-props}. Notice that $\F$-separability may be rephrased as having $\lemb[\F](a)\subseteq \upset a$ for all $a\in L$, and that the reverse set inclusion always holds. Therefore, (2) and (3) are equivalent, too. Finally, if (3) holds then (4) follows immediately. If (4) holds, then for each $a\in L$, the principal filter $\upset a$ is an intersection of filters in $\F$. In particular, by \eqref{eq:concr-e-k}, \(\lemb[\F](a) = \bigcap \{ F\in \F \mid a\in F\} = \upset a\) as it is the intersection of all filters containing $a$.
\end{proof}
From item 5 in Proposition~\ref{p:fe-basic-props}, Proposition~\ref{p:generalchar} and the fact that principal filters are automatically exact and strongly exact we deduce the following.
\begin{corollary}\label{c:inj}
   For a frame $L$ the following maps are always frame embeddings.
   \begin{itemize}
       \item $\lemb[\SE]\colon L\to L^{\SE}$,
       \item $\lemb[\Ex]\colon L\to L^{\Ex}$.
   \end{itemize}
\end{corollary}

Furthermore, the following special cases are also easy to establish. Given our concrete construction of $L\fe$, items 1 and 2 are very close to Theorems 2.2.2 and 3.4.2 in \cite{moshier2022some}. Our proof uses only the universal properties of \(L\fe\).

\begin{proposition}
    The map $\lemb[\F]$ preserves a meet $\bwe_i x_i$ if and only if the filters of $\F$ are closed under that meet, that is, if for all $F\in \F$ we have that $x_i\in F$ for all $i\in I$ implies $\bwe_i x_i$. In particular
    \begin{itemize}
        \item $\lemb[\F]$ preserves strongly exact meets if and only if $\F\subseteq \SE(L)$,
        \item $\lemb[\F]$ preserves exact meets if and only if $\F\subseteq \Ex(L)$,
        \item $\lemb[\F]$ preserves all meets if and only if all filters of $\F$ are principal.
    \end{itemize}
\end{proposition}
\begin{proof}
    The left-to-right direction of all items follows directly from Lemma~\ref{l:ye-pres}.
    Conversely, suppose that $\bwe_ie(x_i)=e(\bwe_i x_i)$. If $\{x_i\}_i\sue F$ for some $F\in \F$ then we get that $\bwe e[F]\leq \bwe_i e(x_i)=e(\bwe_i x_i)$. However, then $\bwe_i x_i\in F$ by \ref{ax:cf}.
\end{proof}

\begin{corollary}
 \label{c:e-pres-e+se-meets}
 We have the following.
 \begin{itemize}
     \item The map $\lemb\colon L\to L^{\SE}$ preserves a meet \(\bwe_i x_i\) if and only if \(\bwe_i x_i\) is strongly exact.
     \item The map $\lemb\colon L\to L^{\Ex}$ preserves a meet \(\bwe_i x_i\) if and only if \(\bwe_i x_i\) is exact.
 \end{itemize}
\end{corollary}

Finally, we look at some simple facts concerning Scott-open and completely prime filters.
\begin{proposition}
    \label{p:lemb-special}
    \hfill
    \begin{enumerate}
        \item If $\F \sue \SO(L)$ then $\lembF$ preserves directed joins.
        \item If $\SO(L) \sue \F$ and \(L\) is pre-spatial (i.e.\ \(\SO(L)\)-separable) then $\lembF$ is injective.
        \item If $\CP(L) \sue \F$ and \(L\) is spatial (i.e.\ \(\CP(L)\)-separable) then $\lembF$ is injective.
    \end{enumerate}
\end{proposition}
\begin{proof}
    (1) We use \ref{ax:df} to show that $\lemb(\bigvee D) \leq \bigvee \lemb(D)$ for any directed $D\sue L$. Let $F\in \F$ be arbitrary, such that $\bwe \lemb[][F] \leq \lemb(\bigvee D)$. By \ref{ax:cf}, $\bigvee D\in F$ and, since \(F\) is Scott-open, there is a $d\in D$ such that $d\in F$. Therefore, $\bwe \lemb[][F] \leq \lemb(d) \leq \bigvee \lemb(D)$ as required. (2) and (3) follow directly from Proposition~\ref{p:fe-basic-props}.6.
\end{proof}

\section{Comparing classes of filters}
\label{s:filt-classes}
In this section, we look at the restrictions of strongly exact filters from the diagram \eqref{eq:fi-inc} in the Introduction.
In particular, we study the joins (intersections) of \emph{regular}, exact, completely prime and Scott-open filters.
Note that all these classes of the form \(L\fe = \J(\F)\), that is, they are closed under joins in \(\Fi(L)\). For exact and strongly exact filters this follows from Lemma~\ref{l:se-e-j} and for the other classes it follows by definition.

Apart from proving the non-obvious inclusions between the classes of filters in the diagram \eqref{eq:fi-inc}, Corollary~\ref{c:all-subcoloc} below shows that these are in fact subcolocale inclusions.

\subsection{Scott-open filters}

We start by remarking that joins of Scott-open filters are strongly exact. The base case of this fact is due to Johnstone (Lemma~3.4 in~\cite{johnstone1985vietoris}, see also \cite{escardo2003joins} and \cite{vickers1997constructive}) who showed that Scott-open filters are strongly exact. 

Johnstone's proof uses a transfinite induction on nuclei, whereas the proof in \cite{escardo2003joins} uses Pataraia's fixed-point theorem for pre-nuclei and \cite{vickers1997constructive} is written in terms of generalised points of powerlocales. We report here a simple explicit proof of the result which, on the other hand, uses Zorn's Lemma.
\begin{lemma}\label{1}
    In a frame $L$ and for $x,y,z\in L$ we have $x\to y=y$ if and only if $z > y$ implies $z \wedge x \nleq y$.
\end{lemma}
\begin{proof}
Note that $x\to y = y$ if and only if $z\nleq y$ implies that $z\wedge x\nleq y$. We claim that this is equivalent to having $z > y$ implying $z \wedge x \nleq y$. It is clear that the first condition implies the second. For the converse, if $z\nleq y$, then we have $z\vee y>y$. Since the second condition holds, we have $(z\vee y)\wedge x=(z\wedge x)\vee(y\wedge x)\nleq y$, and so we must have $z\wedge x\nleq y$.
\end{proof}

\begin{proposition}
    Scott-open filters are strongly exact.
    \label{l:SO-SE}
    \label{p:so-vs-se}
\end{proposition}
\begin{proof}
    Suppose that $F$ is a Scott-open filter of a frame $L$, and that $x\notin F$. Our argument will show that $x$ cannot be a strongly exact meet of elements in F. First, we claim that there must be some $y \in L$ which is maximal among the elements of ${\uparrow}x$ which are not in $F$. If this were
    not the case, then by Zorn's Lemma we could find an infinite ascending chain $x \leq x_1 < ... < x_n < ...$, outside of $F$ but with supremum in \(F\). This would contradict Scott-openness of $F$. Then, let $m \in L$ be such an element, and in particular, observe $\upset m \cap F = \upset m \setminus \{m\}$. Suppose that $x =\bwe_i x_i$ with $\{x_i\}_i\sue F$, we show that the meet cannot be strongly exact. Observe that, for each $i\in I$, and for $z>m$, we must have $z\wedge x_i\nleq m$, because $z\wedge x_i\in F$ and $m\notin F$. Then, by Lemma \ref{1}, $x_i\to m=m$ and, consequently, \(m\in \bigcap_i \os(x_i)\). On the other hand, since $x \leq m$, we have $x \to m= 1 \neq m$ and \(m \notin \os(x)\). Then, $x$ can never be an exact meet of elements in $F$.
\end{proof}

Because strongly exact filters are closed under intersection (Lemma~\ref{l:se-e-j}), Proposition~\ref{p:so-vs-se} entails the desired inclusion from \eqref{eq:fi-inc}.
\begin{corollary}
    \(\Int(\SO(L)) \sue \SE(L)\).
\end{corollary}

\subsection{Completely prime filters}
The class \(\CP(L)\) of completely prime filters is not, in general, closed under intersections. It is clear that we have $\CP(L)\subseteq \SO(L)$. In this special case we can actually describe the filter extension in terms of the dual space \(\pt(L)\) of the frame \(L\). In the following, we regard the points of $\pt(L)$ as completely prime filters \(\CP(L)\).
\begin{lemma}
    \(\Int(\CP(L))\cong \mathcal{U}(\pt(L))\) where, for a topological space \(X\), \(\mathcal{U}(X)\) is the lattice of upsets in the specialisation order of \(X\). The isomorphism is given by
    \[
    P\mapsto \upset P=\{Q\in \CP(L) \mid P\subseteq Q\}.
    \]
\end{lemma}
\begin{proof}
    First, we observe that every \(Q\in \CP(L)\) is completely join prime in \(\Fi(L)\), that is, $Q \sqleq \bigsqcup_i F_i$ (i.e.\ \(\bigcap_i F_i \sue Q\)) implies $Q \sqleq F_i$ (i.e.\ \(F_i \sue Q\)) for some~$i$. Assume by contradiction that, for every \(i\), \(Q \not\sqleq F_i\) i.e.\ there is some \(a_i \in F_i \setminus Q\). Then, \(\bigvee_i a_i \in \bigcap_i F_i \sue Q\) and so \(a_i \in Q\) for some~\(i\), a contradiction.

    Observe that this shows that the map described in the statement is a surjection: any upset $\{P_i \mid i\in I\}$ of completely prime filters is equal to $\{Q\in \CP(L) \mid \bca_i P_i\subseteq Q\}$. From the definition, it is also clear that it is an injection which preserves and reflects order. Therefore, it is an isomorphism.
\end{proof}

\subsection{Exact filters}

The fact that exact filters are a special kind of strongly exact filters
\[
    \Ex(L) \sue \SE(L)
\]
has been observed earlier (cf.\ Remark 3.5 in \cite{moshier2020exact}).
We also make comparison of exact filters with regular filters, which we introduce later on. For that we need the following new characterisation of exact filters.
\begin{proposition}
    \label{p:Ex-new-char}
    For a filter \(F\sue L\) is exact if and only if it is an intersection of filters of the form
    \[
    \upset y \diff \upset x = \{a\in L \mid y\leq a\vee x\}
    \]
    for some $x,y\in L$. In fact, \(F = \bigcap \{ \upset y \diff \upset x \mid \forall f \in F\ y \leq f \vee x\}\) which rewrites as
    \begin{equation}
        F = \{ a \in L \mid \forall x,y\in L \ (\forall f\in F\ y \leq f \vee x) \Rightarrow y \leq a \vee x\}.
        \tag{ex}
        \label{eq:Ex-char}
    \end{equation}
\end{proposition}
\begin{proof}
    From the right-to-left implication assume \(\bigmee_i a_i\) is an exact meet with \(\{a_i\}_i \sue F\). Let \(x,y\) be such that for every \(f\) in \(F\),  \(y \leq f \vee x\). Then, in particular, for every \(i\), \(y \leq a_i \vee x\). Therefore, \(y \leq \bigmee_i (a_i \vee x) = (\bigmee_i a_i) \vee x\) and so \(\bigmee_i a_i\in F\).

    Conversely, denote by \(G\) the right-hand side of \eqref{eq:Ex-char} above. We wish to prove \(F = G\). Observe that \(F \sue G\) is immediate from the definition. For the converse inclusion, let \(a\in G\). In order to show \(a\in F\), it is enough to prove that
    \begin{enumerate}
        \item[(a)] \(a = \bigmee_{f\in F} (f \vee a)\), and
        \item[(b)] the meet in (a) above is exact.
    \end{enumerate}
    For (a) we only need to show \(\bigmee_{f\in F} (f \vee a) \leq a\). Taking \(y = \bigmee_{f\in F} (f \vee a)\) and \(x = a\), observe that we have \(y \leq f \vee a\) for every \(f\in F\). Therefore, since \(a\in G\), we have the desired \(y = \bigmee_{f\in F} (f \vee a) \leq a \vee x = a\).

    For (b) let \(b\in L\) be arbitrary. We need to show that
    \[
        a \vee b = \bigmee_{f\in F} (f \vee a) \vee b \geq \bigmee_{f\in F} (f \vee a \vee b)
    \]
    where the first equality holds by (1). Set \(x = a\vee b\) and \(y = \bigmee_{f\in F} (f \vee a \vee b)\). Observe that trivially, for every \(f'\in F\), \(y \leq f' \vee x\). Therefore, since \(a\in G\), the desired inequality follows: \(y = \bigmee_{f\in F} (f \vee a \vee b) \leq a \vee x = a \vee b\).
\end{proof}

\subsection{Closed and regular filters}
\label{s:Cl-R-filt}

Finally, we introduce a class of filters that (as we explain in Section~\ref{s:closed-open}) corresponds to fittings of closed sublocales. Define
\begin{itemize}
    \item \df{closed filters} \(\Cl(L)\) as the filters of the form \(\{ x \in L\mid a \vee x = 1\}\) for some \(a\in L\).
\end{itemize}
As usual, we order \(\Cl(L)\) by the reverse inclusion order, denoted by \(\sqleq\). It turns out that closed filters are the supplements of principal filters in \(\Fi(L)\):
\begin{align}
    (\upset a)\spp
    = \{1\} \diff \upset a
    = \{ x\in L \mid \forall a' \geq a.\ x \vee a' = 1\}
    = \{ x\in L \mid x \vee a = 1\}.
    \label{eq:cl-supp-open}
\end{align}
The class of intersections of closed filters has a very natural description. Let
\begin{itemize}
    \item \df{regular filters} \(\R(L)\) be the filters of the form \(\{1\} \diff F\) for some filter \(F\sue L\).
\end{itemize}

\begin{lemma}
    \label{l:regular}
    Regular filters coincide with intersections of closed filters:
    \(\R(L) = \Int(\Cl(L)).\)
\end{lemma}
\begin{proof}
    Recall from \eqref{eq:cl-supp-open} that the closed filters are of the form \(\{1\} \diff \upset a\) for some \(a\in L\). Then, for a subset \(A\sue L\), since difference reverts meets in the second coordinate into joins, we obtain that \(\{1\} \diff (\bigsqcap_{a\in A} \upset a) = \bigsqcup_{a\in A} (\{1\} \diff \upset a)\). The result follows from the fact that every filter \(G\in \Fi(L)\) is equal to the meet \(\bigsqcap_{b\in G} \upset b\) in \(\Fi(L)\).
\end{proof}

The \df{Booleanization} \(\B(L)\) of a frame \(L\) is a standard construction in frame theory. It is the sublocale \(\B(L)\sue L\) consisting of all regular elements of \(L\), that is, the elements of the form \(a\to 0\).
When we dualise the situation to coframes, the (co)Booleanization \(\B(C)\) of a coframe \(C\) consists of all supplements \(a\spp\), i.e.\ elements of the form \(1 \diff a\).

In case when the coframe is \(\Fi(L)\), by definition \(\R(L)\) is defined precisely as \(\B(\Fi(L))\).
Therefore, from being a subcolocale, \(\R(L)\) is closed under joins (intersections) in \(\Fi(L)\) and \(\R(L) \cong L^\R\).
Consequently, we obtain the following.

\begin{corollary}
    For a frame \(L\), $L^\Cl$ and $L^\R$ is isomorphic to the Booleanization of $\Fi(L)$.
\end{corollary}

Lastly, observe that, as a consequence of Proposition~\ref{p:Ex-new-char} and the fact that exact filters are closed under intersections, by Lemma~\ref{l:se-e-j}, we have the following inclusions.
\[
    \Cl(L) \ee\sue \R(L) = \J(\Cl(L)) \ee\sue \Ex(L)
\]

\subsection{Subcolocale inclusions of filters}

Recall that we view classes of filters in \eqref{eq:fi-inc} as certain filter extensions. Consequently, we give a general condition for a filter extension \(L\fe\), viewed concretely as \(\J(\F)\) (cf.\ Proposition~\ref{p:fext-concr}), to be a subcolocale of \(\Fi(L)\). Unlike in \cite{jakl20} we do not have to impose any requirements on~\(L\).

We show that, for any class \(\F\) of filters, in order for \(L\fe \sue \Fi(L)\) to be a subcolocale inclusion it is enough for the following following simple property to hold.
\begin{align}
    \forall F \in \F\ \  \forall a \in L,\ F \diff \upset a  = \{ x \mid x \vee a \in F\} \in \F.
    \tag{scl}
    \label{ax:subco}
\end{align}

Note that all classes of filters that we've looked at so far satisfy this property.

\begin{lemma}
    \label{l:scl-holds}
    $\F$ satisfies \eqref{ax:subco} if it is the class of
    \begin{enumerate}
        \item closed filters $\Cl(L)$,
        \item Scott-open filters $\SO(L)$,
        \item completely prime filters $\CP(L)$,
        \item exact filters $\Ex(L)$, or
        \item strongly exact filters $\SE(L)$.
    \end{enumerate}
\end{lemma}
\begin{proof}
    (1) If $F_d = \{ x \mid x \vee d = 1\}$ is a closed filter, then $F_d \setminus \upset c = \{ x \mid x \vee c \in F\} = \{ x \mid x\vee c \vee d = 1 \} = F_{c\vee d}$ is a closed filter too.
    (2) If $F$ is Scott-open and a directed join \(\bigvee_i a_i\) is in \(F \setminus \upset c\) then the directed join $\bigvee_i (a_i \vee c) = (\bigvee_i a_i) \vee c$ is in $F$. Hence, for some $i$, $a_i \vee c$ is in $F$, giving that $a_i \in F \setminus \upset c$.
    (3) is proved analogously to (2).

    (4) Let $F$ be an exact filter and assume that $\{ x_i \}_i \sue F\setminus \upset c$ has an exact meet. Observe that $\bigmee_i x_i \in F\setminus \upset c$ because, for every \(i\), $x_i \vee c \in F$ and so $(\bigmee_i x_i) \vee c = \bigmee_i (x_i \vee c)\in F$ since $F$ is exact and $\bigmee_i (x_i \vee c)$ is an exact meet.

    (5) Proposition 3.4 in \cite{moshier2020exact} says that if $\{x_i\}_i$ has a strongly exact meet then so has $\{x_i \vee c\}_i$, for any $c$, and also $\bigmee_i (x_i \vee c) = (\bigmee_i x_i) \vee c$. The rest of the proof goes as for (4).
\end{proof}

The following fact about the closure of a subset of a coframe under joins is needed in the theorem below.

\begin{lemma}
    Assume that \(C\) is a coframe and \(S\sue C\) satisfies
    \begin{align}
        \forall c\in C\ \forall s\in S,\ s \diff c \in \J(S)
        \label{eq:scl-abstract}
    \end{align}
    then \(\J(S) \sue C\) is a subcolocale inclusion.
    \label{l:scl-simpl}
\end{lemma}
\begin{proof}
    Recall that coHeyting difference \(\diff\) preserves joins in the first component. Hence, \(\J(S)\) is stable under \((\ARG) \diff c\) since, for any \(A\sue S\),
    \[
        (\bigvee A) \diff c = \bigvee \{ a \diff c \mid a\in A\}
    \]
    and this is in \(\J(S)\) by \eqref{eq:scl-abstract}.
\end{proof}

With this we easily see that \eqref{ax:subco} suffices for \(L\fe\) to be a coframe. In fact we have more, it is a subcolocale of \(\Fi(L)\).

\begin{theorem}
    \label{t:scl-main}
    For a frame \(L\), if a collection \(\F \sue \Fi(L)\) satisfies \eqref{ax:subco} then the inclusion \(L^\F \sue \Fi(L)\) is a subcolocale inclusion.
\end{theorem}
\begin{proof}
    Recall that coHeyting difference reverses meets into joins in the second component. Then, since joins of filters are intersections, for any \(G\in \Fi(L)\) and \(F\in \F\),
    \[
        F \diff G = \bigcap \{ F \diff \upset a \mid a\in G\}.
    \]
    Therefore, by \eqref{ax:subco} the assumptions of Lemma~\ref{l:scl-simpl} are met for joins of \(\F\) in \(\Fi(L)\), i.e.\ for \(L^\F = \Int(\F)\).
\end{proof}

Now by Lemma~\ref{l:scl-holds} and Theorem~\ref{t:scl-main} we see that all classes of filters appearing in the diagram in \eqref{eq:fi-inc} are subcolocales of \(\Fi(L)\). Furthermore, upon recalling that if \(S, T \sue C\) are subcolocales of a coframe \(C\) and \(S \sue T\) then the inclusion \(S \sue T\) is a subcolocale inclusion as well, we obtain the following.

\begin{corollary}
    All classes of filters in \eqref{eq:fi-inc} are coframes and, furthermore, all inclusions therein are subcolocale inclusions.
    \label{c:all-subcoloc}
\end{corollary}

\subsection{Topological properties and classes of filters}

We now show some characterizations of frame theoretical properties in terms of the collections of filters we have discussed so far.

Recall that a frame \(L\) is \df{subfit} if \(a\leq b\) whenever, for every \(c\in L\), \(a\vee c = 1\) implies \(b \vee c = 1\).

\begin{proposition}\label{famouschar}
For a frame $L$ we have the following.
\begin{itemize}
    \item $L$ is pre-spatial (i.e.\ \(\SO(L)\)-separable) iff $\J(\SO(L))$ contains all principal filters.
    \item $L$ is spatial (i.e.\ \(\CP(L)\)-separable) iff $\J(\CP(L))$ contains all principal filters.
    \item $L$ is subfit iff $\R(L)$ contains all principal filters.
\end{itemize}
\end{proposition}
\begin{proof}
    Notice that all three statements are special cases of Proposition~\ref{p:generalchar}, with $\F$ chosen to be the collection $\SO(L)$ in the first case, the collection $\CP(L)$ in the second, and the collection of filters of the form $\{x\in L:x\vee a=1\}$ in the third. Indeed, we have \(\J(\Cl(L)) = \R(L)\) by Proposition~\ref{l:regular}.
\end{proof}

\begin{proposition}\label{p:sfre}
    A frame is subfit if and only if all exact filters are regular.
\end{proposition}
\begin{proof}
First, suppose that $L$ is a subfit frame. By the characterization in Proposition~\ref{p:Ex-new-char}, it suffices to show that filters of the form 
\[
\{x\in L \mid b\leq x\vee a\}
\]
are equal to the intersection of the closed filters above them. Consider, then, a filter $F$ of the form above. Suppose that we have $z\in L$ such that $z\in C$ whenever $C$ is a closed filter such that $F\subseteq C$. We show that $z\in F$. In order to show $b\leq z\vee a$, by subfitness, it suffices to show that $b\vee u=1$ implies that $z\vee a\vee u=1$ for all $u\in L$. Suppose, then, that the antecedent holds. Consider the closed filter $C=\{x\in L:x\vee a\vee u=1\}$. We claim that $F\subseteq C$. If $b\leq x\vee a$, then because we assumed the antecedent above we have $x\vee a\vee u=1$ and so $x\in C$. So, indeed $F\subseteq C$. By our assumptions, this means $z\in C$ or, in other words, \(z \vee a \vee u = 1\).

If, conversely, all exact filters are regular, in particular all principal filters are, and the frame is subfit by item (3) of Proposition \ref{famouschar}.
\end{proof}

We obtain a proof of Theorem~3.2 of \cite{ball2020exact} and its converse, which can be thought of as the filter-only version of Theorem~3.5 in \cite{picado2019Sc}, without the detour to sublocales.

\begin{corollary}
    A frame is subfit if and only if $\Ex(L)$ is Boolean.
\end{corollary}
\begin{proof}
    We have seen above that all regular filters are exact for any frame. If $L$ is subfit, by Proposition \ref{p:sfre} we also have $\Ex(L)=\R(L)$, and so $\Ex(L)$ is Boolean. Conversely, if $\Ex(L)$ is Boolean, $\R(L)\subseteq \Ex(L)$ implies that $\R(L)=\Ex(L)$, as the Booleanization is maximal among the Boolean sub(co)locales.
\end{proof}

\section{Relating fitted sublocales and filters}
\label{s:relating}

In this section we connect the constructions of Galois closed sets arising from polarities of filters with recent important sublocale-based constructions of discretisations of a frame.
A common way (see e.g.\ in Section~4.2 of \cite{moshier2020exact}) to connect filters and sublocales of a frame is via the adjunction
\begin{equation}
    \begin{tikzcd}[bend left]
        \Fi(L)
            \rar{\fts}
            \rar[bend left=0,phantom,pos=0.43]{\top}
        &
        \Sl(L)
            \lar{\stf}
    \end{tikzcd}
    \label{eq:fi-su-adj}
\end{equation}
where \(\stf(S) = \{ a\in L \mid S \sue \os(a)\}\) and \(\fts\) is defined as in~\eqref{eq:fi-to-sl}.
It is immediate to see that \(\stf \dashv \fts\), that is, \(\stf\) is the left adjoint to \(\fts\): \(\stf(S) \sqleq F\) iff \(S \sue \fts(F)\).

In the following we discuss how restrictions of this adjunction to different classes of sublocales, on the right, lead to restrictions to different classes of filters, on the left.

\subsection{Fitted sublocales}
\label{s:SE}

First we connect our constructions with a prominent recent result representing fitted sublocales as strongly exact filters. In \cite{moshier2020exact} it is shown that the adjunction in \eqref{eq:fi-su-adj} restricts to the isomorphism between strongly exact filters and
\begin{itemize}
    \item \df{fitted sublocales} \(\So(L)\), that is, sublocales of the form \(\bigcap_{m\in M} \os(m)\) for some \(M\sue L\),
\end{itemize}

The proof of the correspondence in \cite{moshier2020exact} is not complicated. However, we offer a different one which makes use of the theory of polarities. We start with an observation.

\begin{lemma}
    The mapping \(\stf\) restricts to \(\So(L) \to \SE(L)\) and, moreover, for \(S \in \So(L)\) we have \(\fts(\stf(S))=S\).
    \label{l:se-basics}
\end{lemma}
\begin{proof}
    Let \(S\sue L\) be a sublocale of the form \(\bigcap_{m\in M} \os(m)\) for some \(M\sue L\).
    Since \(M \sue \stf(S)\) we have \(\fts(\stf(S)) \sue S\). The fact that \(\stf\) is the left adjoint to \(\fts\) implies the other inequality, that is, \(S = \fts(\stf(S))\).

    With this and Lemma~\ref{l:SE-char} we also show that \(\stf(S)\) is strongly exact. For any \(b\in L\) such that \(\bigcap_{a\in \stf(S)} \os(a) \sue \os(b)\) since the left-hand-side is equal to \(S\) we also have that \(S \sue \os(b)\). Therefore, \(b\) is in \(\stf(S)\).
\end{proof}

With these we immediately get the main result of \cite{moshier2020exact}.

\begin{theorem}
    \label{t:SE-iso}
    The adjunction in \eqref{eq:fi-su-adj} restricts to the isomorphism \(\SE(L) \cong \So(L)\).
\end{theorem}
\begin{proof}
    Consider the mapping \(\os\colon L \to \So(L)\). Clearly, the image \(\bigmee\)-generates \(\So(L)\) and, by the second part of Lemma~\ref{l:se-basics}, the elements of the form \(\bigmee \os[F]\) for strongly exact filter \(F\) \(\bigvee\)-generates \(\So(L)\). Finally, by \eqref{eq:se-fi}, we have that \(\fts(F) \sue \os(a)\) iff \(a\in F\) iff \(F Z a\). Hence, \(\os \colon L \to \So(L)\) satisfies the axioms \ref{ax:df} and \ref{ax:cf} for \(\F = \SE(L)\).
    As a result \(\SE(L) = \Int(\SE(L)) \cong L^\SE\cong \So(L)\) by Proposition~\ref{p:fext-concr} and Lemma~\ref{l:SE-intersect}.
\end{proof}

\paragraph{Fitted joins.}
\label{par:fit-join}
The lattice \(\So(L)\) can be presented as the sublattice of \(\Sl(L)\) of fixpoints of the \df{fitting operator}
\[\fit\colon \Sl(L) \to \Sl(L),\]
which sends a sublocale \(S\) to \(\bigcap \{\os(a) \mid S \sue \os(a) \}\). Since this is a closure operator, joins and meets in \(\So(L)\) can be computed in \(\Sl(L)\) according to \eqref{eq:fix-cl} as follows.
\begin{equation}
    \textstyle
    \bigmee^{\So(L)} A \ee{=} \bigmee^{\Sl(L)} A \ee= \bigcap A
    \qtq{and}
    \bigfitvee A \ee{=} \fit(\bigvee^{\Sl(L)} A).
    \label{eq:So-lattice}
\end{equation}
To distinguish the joins in \(\So(L)\) from those in \(\Sl(L)\), we call the former ones \df{fitted joins}.
Then, for a set \(M \sue \So(L)\), we write
\[
    \FJ(M) = \{ \bigfitvee A \mid A \sue M \}
\]
for the closure of \(M\) under fitted joins in \(\So(L)\).

\subsection{Smooth sublocales}

In \cite{picado2019Sc}, the authors introduced the lattice
\[\Sc(L) \sue \Sl(L)\]
consisting \df{joins of closed sublocales}, i.e.\ sublocales of the form \(\bigvee_{a\in M} \cs(a)\) for some \(M\sue L\). Among others, \cite{picado2019Sc} shows that that strongly exact filters correspond to \(\Sc(L)\). In fact, Theorem~\ref{t:SE-iso} above is inspired by this fact. However, the isomorphism \(\Ex(L) \cong \Sc(L)\) takes \(\Ex(L)\) in the usual subset order \(\sue\).

Therefore, the isomorphism from \cite{picado2019Sc} cannot be obtained as a restriction of Theorem~\ref{t:SE-iso} along the inclusion \(\Ex(L) \sue \SE(L)\).
In fact, in \cite{moshier2020exact} it is proven that this restriction gives us fittings of supplements of joins of closed sublocales on the other side, i.e. Theorem~\ref{t:SE-iso} restricts to $\Ex(L) \cong \fit[(\Sc(L))\spp]$.
In the following lines we show that sublocales in \(\fit[(\Sc(L))\spp]\) are precisely fittings of smooth sublocales, from \cite{arrieta2021complemented}.

Observe that the filters arising in Proposition \ref{p:Ex-new-char} are the filters associated with \df{locally closed sublocales}, i.e.\ sublocales of the form \(\cs(x) \cap \os(y)\). Indeed,
\begin{equation}
    \stf(\cs(x) \cap \os(y)) = \{ a \mid y \leq a \vee x \} = \upset y \diff \upset x
    \label{eq:stf-smooth-base}
\end{equation}
since \(\cs(x) \cap \os(y) = \os(y) \diff \os(x) \sue \os(a)\) iff \(\os(y) \sue \os(a) \vee \os(x) = \os(a\vee x)\). Therefore, it is justified calling filters of the form \(\upset y \diff \upset x\) \df{locally closed filters}. We denote by 
\[
    \LCl(L)
    \qtq{and}
    \Slc(L)
\]
the class of locally closed filters and locally closed sublocales, respectively. By the same proof as in the first paragraph of Proposition~\ref{p:Ex-new-char} we see that \(\LCl(L) \sue \Ex(L)\). Therefore, recalling~\eqref{eq:int-F} Proposition~\ref{p:Ex-new-char} immediately yields.
\begin{lemma}
    \label{l:ex-lcl}
    Exact filters are the joins of locally closed filters, i.e.\ \(\Ex(L) = \J(\LCl(L))\).
\end{lemma}

On the other hand, since locally closed filters are in particular strongly exact filters and since the adjunction \(\stf \dashv \fts\) restricts to an isomorphism on strongly exact filters, the description of locally closed filters from~\eqref{eq:stf-smooth-base} immediately yields the following.

\begin{lemma}
    The isomorphism in Theorem~\ref{t:SE-iso} restricts to \(\LCl(L) \cong \fit[\Slc(L)]\).
    \label{l:LCl-iso}
\end{lemma}

Observe that the isomorphism in Theorem~\ref{t:SE-iso} translates joins of strongly exact filters to fitted joins of the fitted sublocales. Therefore, by the previous two lemmas, exact filters correspond precisely to fitted joins of the fittings of locally closed sublocales. The following significantly simplifies the description of sublocales arising this way.

\begin{lemma}
    Given a class of sublocales \(\mathcal A \sue \Sl(L)\), fitted joins of fittings of sublocales in \(\mathcal A\) coincide with fittings of joins of sublocales in \(\mathcal A\), i.e.\ \(\FJ(\fit[\mathcal A]) = \fit[\J(\mathcal A)]\).
    \label{l:FJ-fittings}
\end{lemma}
\begin{proof}
    Given \(\{S_i\}_i \sue \mathcal A\), the fitted join \(\bigfitvee[L][i] \fit(S_i)\) is just the meet of sublocales \(\fit(\bigvee\nolimits_i\, \fit(S_i)) = \bigcap \{ \os(a) \mid \bigvee\nolimits_i\, \fit(S_i) \sue \os(a)\}\) and we have that
    \begin{align*}
        \bigvee\nolimits_i\, \fit(S_i) \sue \os(a)
        &\iff \forall i\ S_i \sue \os(a)
        \\
        &\iff \bigvee\nolimits_i\,  S_i \sue \os(a)
        \iff \fit(\bigvee\nolimits_i\,  S_i) \sue \os(a).
        \qedhere
    \end{align*}
\end{proof}

Combining all the above we obtain that exact filters precisely correspond to fittings of joins of locally closed sublocales. The latter class of sublocales has been studied in~\cite{arrieta2021complemented}. In fact, the sublocales of the form \(\textstyle S = \bigvee_i \, (\cs(x_i) \cap \os(y_i))\) for some set of \(x_i\)'s and \(y_i\)'s, are called \df{smooth sublocales} and correspond precisely to joins of complemented sublocales.
Upon denoting the collection of smooth sublocales by \(\Sb(L)\), we can conclude the following.

\begin{theorem}
    The isomorphism in Theorem~\ref{t:SE-iso} restricts to \(\Ex(L) \cong \fit[\Sb(L)]\).
    \label{t:Ex-iso}
\end{theorem}

\subsection{Joins of closed sublocales}

The results from the previous section, especially Theorem~\ref{t:Ex-iso}, naturally lead to further questions about the nature of the frame \(\Sc(L)\) of joins of closed sublocales. Observe that we have \(\Sc(L) \sue \Sb(L)\) and so there must be a restriction of \(\Ex(L)\) that corresponds to \(\fit[\Sc(L)]\). To this end, we appeal to the notion of closed filters which we have introduced in Section~\ref{s:Cl-R-filt}. The reason why we call these filter closed is because they are of the form \(\stf(\cs(a))\):
\begin{align}
    \stf(\cs(a))
    = \{ x\in L \mid \cs(a) \sue \os(x)\}
    = \{ x\in L \mid \Fll \diff \os(a) \sue \os(x)\}
    = \{ x\in L \mid x \vee a = 1\}
    \label{eq:cl-filt}
\end{align}

Wee see that closed filters are a special kind of locally closed filters, which in turn are a subclass of exact filters
\[\Cl(L) \sue \LCl(L) \sue \Ex(L).\]
Consequently, the following theorem is a genuine refinement of Theorem~\ref{t:Ex-iso}.

\begin{proposition}
    \label{p:Cl-iso}
    The isomorphism in Theorem~\ref{t:SE-iso} restricts to \(\Int(\Cl(L)) \cong \fit[\Sc(L)]\).
\end{proposition}
\begin{proof}
    By \eqref{eq:cl-filt} we know that a filter is closed if and only if it is the \(\stf(\ARG)\) image of a closed sublocale. Therefore, intersections of closed filters \(\Int(\Cl(L))\) are isomorphic to fitted joins of fittings of closed sublocales by Theorem~\ref{t:Ex-iso}.

    What is left to show is that these are precisely fittings of joins of closed sublocales. This follows from Lemma~\ref{l:FJ-fittings}.
\end{proof}

Furthermore, Lemma~\ref{l:regular} gives the following rephrasing of Proposition~\ref{p:Cl-iso}.

\begin{theorem}
    \label{t:Cl-iso}
    The isomorphism in Theorem~\ref{t:SE-iso} restricts to \(\R(L) \cong \fit[\Sc(L)]\).
\end{theorem}

\subsection{Joins of compact sublocales}

The celebrated Hofmann-Mislove Theorem establishes a correspondence between compact saturated subsets of a sober space and Scott-open filters on the frame of open of the space. Johnstone gave a pointfree analogue of this result.

\begin{proposition}
    \label{t:johnstone}
    The isomorphism in Theorem~\ref{t:SE-iso} restricts to the isomorphism between \(\SO(L)\) and \df{compact fitted sublocales} of \(L\), that is, sublocales \(S\in \So(L)\) such that \(S \sue \bigvee_{a\in A} \os(a)\) implies \(S \sue \os(a_1) \vee \dots \vee \os(a_n)\) for some \(a_1,\dots,a_n\in A\)..
\end{proposition}
\begin{proof}
    This is Lemma~3.4 in~\cite{johnstone1985vietoris}, see also~\cite{escardo2003joins}.
\end{proof}

This theorem gives an alternative (and constructive) proof of Lemma~\ref{l:SO-SE} that \(\SO(L) \sue \SE(L)\). Indeed, since \(\SE(L)\) are the fixpoints of the adjunction \eqref{eq:fi-su-adj}, by Proposition~\ref{t:johnstone}, \(\SO(L)\) is in the image of the mapping \(\stf\).

However, unlike exact and strongly exact filters, Scott-open filters are not necessarily closed under intersection. Therefore, the arising filter extension \(L^\SO = \Int(\SO(L))\) induces a non-trivial extension of Johnstone's theorem.
\begin{proposition}
    The isomorphism in Theorem~\ref{t:SE-iso} restricts to an isomorphism between \(\Int(\SO(L))\) and fitted joins of compact fitted sublocales.
\end{proposition}

This proposition already gives an already quite intuitive topological description of the sublocales corresponding to filters in \(\Int(\SO(L))\). However, we can specify these sublocales in another slightly simpler way. To this end, set
\begin{itemize}
    \item \(\Sk(L)\) to be the poset of joins of compact sublocales in \(\Sl(L)\).
\end{itemize}

\begin{theorem}
    \label{t:SO-iso}
    The isomorphism in Theorem~\ref{t:SE-iso} restricts to \(\Int(\SO(L)) \cong \fit[\Sk(L)]\).
\end{theorem}
\begin{proof}
    First, observe that compact fitted sublocales are exactly the fittings of compact sublocales.
    The inclusion from left to right is immediate. For the other direction, observe that $S\sue \os(a)$ if and only if $\fit(S)\sue \os(a)$ for any sublocale $S$ and $a\in L$. Therefore, if a sublocale is compact then its fitting is compact too.

    From this and Lemma~\ref{l:FJ-fittings} we see that every fitted join of compact fitted sublocale is in fact of the form \(\fit(\bigvee\nolimits_i K_i)\) for some compact \(K_i\)'s.
\end{proof}

\subsection{Spatial sublocales}

In our final refinement of the adjunction \eqref{eq:fi-su-adj} we further restrict Theorem~\ref{t:SO-iso} along the inclusion \(\CP(L)\sue \SO(L)\) and describe the sublocales corresponding to the class of completely prime filters.

To this end recall, e.g.\ from \cite{arrieta2023spatial} or IX.3.3 in \cite{pp2021book2}, that an \(p\in L\) is \df{prime} iff \(\{p,1\}\) is a sublocale of \(L\). When this is the case, we call \(\{p,1\}\) a \df{one-point sublocale} and denote it by \(\bs(p)\).

Moreover, we can define the \df{spatialization operation} on sublocales:
\[
    \spa:\SL \to \SL,
    \qquad
    S \sue L
    \ee\mapsto
    \bigvee \{ \bs(p) \mid p\in \pt(L),\ \bs(p) \sue S\}
\]
With this, define
\begin{itemize}
    \item \(\Ssp(L)\) as the poset of \df{spatial sublocales} of \(L\), that is, sublocales \(S \sue L\) such that \(S = \spa(S)\).
\end{itemize}
Observe that \(\spa\) is an interior operator which preserves finite joins. As a consequence, $\spa$ is a coframe map \(\SL\to \Ssp(L)\) and therefore $\Ssp(L)\sue \SL$ is a subcolocale inclusion. On the other hand, the following theorem is crucial in establishing that \(\fit[\Ssp(L)] \sue \So(L)\) is a subcolocale of \(\So(L)\) in Section~\ref{s:subcolocales} below.

\begin{theorem}
    \label{t:CP-iso}
    The isomorphism in Theorem~\ref{t:SE-iso} restricts to \(\Int(\CP(L)) \cong \fit[\Ssp(L)]\).
\end{theorem}
\begin{proof}
    First, observe that for a prime \(p\in L\),
    \begin{equation}
        \bs(p)\sue \os(a)
        \ee\iff
        a \nleq p
        \label{eq:blp-os}.
    \end{equation}
    Indeed, \(\bs(p)\sue \os(a)\) iff \(a \to p = p\) iff \(a \nleq p\) where the last equivalence holds because \(a \to b = 1\) iff \(a\leq b\) and because \(p \not= 1\) since \(p\) is prime.

    With this, we show that a filter \(F\) is completely prime if and only if it is $\stf(\bs(p))$ for some prime $p\in L$. For the left-to-right direction, suppose that we have a prime $p\in \pt(L)$ and that \(\bve_i x_i \in \stf(\bs(p))\), that is, $\bs(p)\sue \os(\bve_i x_i)$.
    Then, by \eqref{eq:blp-os}, $\bve_i x_i\nleq p$, which implies that $x_j\nleq p$ for some $j\in I$, that is, $\bs(p)\sue \os(x_j)$. Therefore the filter $\stf(\bs(p))$ is completely prime.

    For the converse, suppose that $F\sue L$ is a completely prime filter.
    Consider the prime element $p_F=\bve (L\sm F)$.
    We clearly have that $f\nleq p_F$ for all $f\in F$ which, by \eqref{eq:blp-os}, implies $F\sue \stf(\bs(p_F))$. Finally, notice that, by definition of $p_F$, $f\nleq p_F$ implies $f\in F$.

    This gives us a poset isomorphism \(\CP(L) \cong \{ \bs(p) \mid p\in \pt(L)\}\). Because $\fts$ restricted to \(\SE(L)\) is an isomorphism of complete lattices (by Theorem~\ref{t:SE-iso}) and because \(\CP(L) \sue \SE(L)\) (by Lemma~\ref{l:SO-SE}), intersections of completely prime filters correspond to fitted joins of one-point sublocales. Furthermore, since joins of one-point sublocales are precisely spatial sublocales, we obtain that \(\FJ(\{ \bs(p) \mid p\in \pt(L)\}) = \fit[\Ssp(L)]\) as required.
\end{proof}

\subsection{The big picture}
\label{s:subcolocales}

Corollary~\ref{c:all-subcoloc} together with Theorems \ref{t:SE-iso}, \ref{t:Ex-iso}, \ref{t:Cl-iso}, \ref{t:SO-iso}, and \ref{t:CP-iso} allow us to rephrase the subcolocale inclusions of classes of filters from \eqref{eq:fi-inc} as the following inclusion of classes of fitted sublocales.
\begin{equation}
    \begin{tikzcd}[row sep=1em, column sep=1.5em]
        \fit[\Sc(L)] \rar[hook] & \fit[\Sb(L)] \ar[hook]{rd} \\
            && \So(L) \\
        \fit[\Ssp(L)] \rar[hook] & \fit[\Sk(L)] \ar[hook]{ru}
    \end{tikzcd}
    \label{eq:all-subcoloc}
\end{equation}
Recall that \(\So(L)\) itself is a sub-coframe of \(\SL\) \cite{moshier2020exact}. Some of these results were already covered in the literature.
In particular, the lattices \(\So(L)\) and \(\Sc(L)\op \cong \fit[\Sb(L)]\) are known to be coframes \cite{picado2019Sc,moshier2020exact}.  Whilst relying on non-trivial facts from \cite{picado2019Sc}, it is proved in \cite{moshier2020exact} that the sequence \(\Ex(L) \sue \SE(L) \sue \Fi(L)\) is a sequence of subcolocale inclusions.  Separately, it is also known that \(\fit[\Sk(L)]\) is a coframe when \(L\) is spatial or stably locally compact \cite{jakl20}. And finally, in an unpublished note, Joshua L. Wrigley has observed that \(L^\CP\) is a coframe too (private communication).

Furthermore, if a frame is \df{fit}, i.e.\ if all sublocales are fitted, we obtain the simplification \eqref{eq:sl-inc} of \eqref{eq:all-subcoloc} from the Introduction. Note that \(\Sc(L) = \Sb(L)\) follows from Proposition~\ref{p:sfre} and the fact that every fit frame is subfit.

\section{Immediate consequences of the theory of polarities}

In Section~\ref{s:relating}, we established a number of correspondences between classes of filters closed under intersections and classes of sublocales. In fact, since \(L^\F \cong \J(\F)\) (cf.\ Proposition~\ref{p:fext-concr}), the main theorems in the previous section can be rephrased as
\begin{equation}
    L^\F \cong \fit[\mathcal C]
    \label{eq:fe-vs-fit-C}
\end{equation}
for some class of filters \(\F\sue \Fi(L)\) and a class of sublocales \(\mathcal C \sue \Sl(L)\). Recalling that the filter extension \(L^\F\) is given as \(\GC(P)\) for the polarity of the form \(P = (\F,L,\ni)\) (cf.\ Section~\ref{s:filter-ext}), in this section we make use of the connections with the theory of polarities to obtain new facts about filters and sublocales.

\subsection{Isomorphisms with interiors of fitted sublocales}
Observe that the operator \(\fit(\ARG)\) on sublocales is simply the closure operator \(\clOp_{\os[L]}(\ARG)\) for the class \(\os[L]\) of open sublocales, in sense of \S\ref{par:inside-clat}. Furthermore, the classes of sublocales in \eqref{eq:fe-vs-fit-C} in our examples are closed under joins in \(\Sl(L)\). In fact they are of the form \(\J(\mathcal D)\) for some smaller subclass of sublocales \(\mathcal D\) of \(\Sl(L)\). Namely,

\begin{itemize}
    \item \(\Sb(L) = \J(\Slc(L))\) where \(\Slc(L)\) is the class of locally closed sublocales,
    \item \(\Sc(L) = \J(\cs[L])\) where \(\cs[L] = \{ \cs(a) \mid a\in L\}\) is the class of closed sublocales,
    \item \(\Sk(L) = \J(\Sco(C))\) where \(\Sco(L)\) is the class of compact sublocales, and
    \item \(\Ssp(L) = \J(\Sop(L))\) where \(\Sop(L)\) is the class of one-point sublocales, i.e.\ sublocales \(\bs(p)\) for some prime \(p \in L\).
\end{itemize}

This immediately puts us in the scope of Proposition~\ref{p:int-cl} about polarities, which reveals interesting new correspondences with the classes of sublocales defined dually.

\begin{theorem}
    \label{t:fit-vs-int}
    For a frame \(L\),
    \begin{enumerate}
        \item fittings of smooth sublocales correspond to locally closed interiors of fitted sublocales,\\ i.e.\ \(\fit[\Sb(L)] \cong \inOp_{\Slc(L)}[\So(L)]\),
        \item fittings of joins of closed sublocales correspond to closed interiors of fitted sublocales,\\ i.e.\ \(\fit[\Sc(L)] \cong \inOp_{\cs[L]}[\So(L)]\),
        \item fittings of joins of compact sublocales correspond to compact interiors of fitted sublocales, i.e.\ \(\fit[\Sk(L)] \cong \inOp_{\Sco(L)}[\So(L)]\), and
        \item fittings of spatial sublocales correspond to spatializations of fitted sublocales,\\ i.e.\ \(\fit[\Ssp(L)] \cong \inOp_{\Sop(L)}[\So(L)] = \spa[\So(L)] \).
    \end{enumerate}
    The isomorphisms in each of the cases are described as in \$\ref{par:iso-cl-int}.
\end{theorem}

As we show in Section~\ref{s:subcolocales}, all these classes of sublocales form a coframe. Before we get to that we explore the universal properties that these classes admit, thanks to the theory of polarities.

\subsection{Universal properties}
\label{s:subloc-univ-prop}

Coming back to \eqref{eq:fe-vs-fit-C}, we see that Theorem~\ref{t:fit-vs-int} tells us that our example filter extensions can be equivalently specified by a class of sublocales instead.
Indeed, for a strongly exact filter \(F\sue L\) and \(a\in L\), \(\fts(F) \sue \os(a)\) iff \(a\in F\). This gives us that the mapping \(\fts\colon \Fi(L) \to \Sl(L)\), when restricted to strongly exact filters, is order reflective and also that, for a class \(\F \sue \SE(L)\),
\[L^\F \cong \GC(\fts[\F], \os[L], \sue).\]
Therefore, by Propositions~\ref{p:int-cl}, \ref{p:fext-concr} and Theorems \ref{t:Ex-iso}, \ref{t:Cl-iso}, \ref{t:SO-iso}, and \ref{t:CP-iso}, we obtain a corollary. 
\begin{corollary}
    For a frame \(L\), we have:
    \begin{align*}
        L^\Ex \cong L^\LCl &\cong \GC(\Slc(L), \os[L], \sue) 
        && L^\Cl \cong \GC(\cs[L], \os[L], \sue)\\
        L^\SO &\cong \GC( \Sco(L), \os[L], \sue) &
        L^\CP &\cong \GC(\Sop(L), \os[L], \sue)
    \end{align*}
\end{corollary}

\noindent
Note that the first isomorphism \(L^\Ex \cong L^\LCl\) follows from Lemma~\ref{l:ex-lcl}.

As a result, we can reason about these classes of sublocales in terms of filters and vice versa. Take for example the filter extension \(L^\Ex\). The induced mapping
\begin{equation}
    L \xrightarrow{\ee\lemb} L^\Ex
    \label{eq:L-fe-Ex-univ}
\end{equation}
satisfies the universal properties \ref{ax:df} and \ref{ax:cf} for \(\F = \Ex(L)\) and, by Corollary~\ref{c:e-pres-e+se-meets}, preserves exact meets.
Now, by Propostion~\ref{p:int-cl}, we know that
\[L^\Ex \,\cong\, \GC(\Slc(L), \os[L], \sue) \,\cong\, \fit[\Sb(L)] \,\cong\, \inOp_{\Slc(L)}[\So(L)].\]
We can now look at the  concrete description of \(\lemb\) from \eqref{eq:L-fe-Ex-univ}, as given in Section~\ref{s:concrete-desc}, and compose it with the isomorphism from the proof of Proposition~\ref{p:int-cl} and also \S\ref{par:iso-cl-int}. We obtain maps
\[
    L \longrightarrow \inOp_{\Slc(L)}[\So(L)]
    \qtq{and}
    L \longrightarrow \fit[\Sb(L)]
\]
with the first one given by \(a\mapsto \bigvee\{ \cs(x) \cap \os(y) \mid \cs(x) \cap \os(y) \sue \os(a)\}\) and the second one as a fitting of the result of the first. As a consequence, the same properties that hold for \(e\) in \eqref{eq:L-fe-Ex-univ} also hold for these maps. In particular, these two maps preserve exact meets.

In a completely analogous way the induced mappings
\[
    L \to L^\SO
    \qquad
    L \to L^\Cl
    \qquad
    L \to L^\CP
    \qquad
    L \to L^\SE
\]
satisfy \ref{ax:df} and \ref{ax:cf} with \(\F\) taken to be the closed, completely prime and strongly exact filters, respectively.
Repeating the same procedure as above gives us that the mapping \(\os \colon L \to \So(L)\) obtained from \(L \to L^\SE\) preserves strongly exact meets by Corollary~\ref{c:e-pres-e+se-meets}. This is precisely the Technical Lemma in \cite{moshier2020exact}.

\subsection{Exact filters in the subset order}

Recall that \cite{ball2020exact} shows that \((\Ex(L), \sue) \cong \Sc(L)\) (note the use of the subset order for exact filters). By Corollary~\ref{c:GC-op} we can describe the universal property giving \((\Ex(L), \sue)\) in terms of the one describing \((\Ex(L),\sqleq)\):
\begin{align*}
    (\Ex(L),\sue)
    \cong (\Ex(L), \sqleq)\op
    \cong \GC(\Ex(L), L, \ni)\op
    \cong \GC(L, \Ex(L), \in)
\end{align*}
Note that the natural order of \(L\) and \(\Ex(L)\) in \((L, \Ex(L), \in)\) are dual to what we are used to. Indeed, if \(a \geq b\) and \(b\in F\) for some \(F \sue G\) from \(\Ex(L)\) then \(a\in G\). This way the mappings into \(\GC(L, \Ex(L), \in)\) stay monotone by Lemma~\ref{l:xe-ye-basics}.

Consequently, consider the following pair of maps
\[
    \begin{tikzcd}
        L\op \rar{\cs}
        & \Sc(L)
        & (\Ex(L),\sue) \lar[swap]{\ftsC}
    \end{tikzcd}
\]
where \(J\) sends an exact filter \(F\) to \(\bigvee_{f\in F} \cs(f)\). In fact, \(J\) is the isomorphism of \(\Sc(L)\) and \((\Ex(L),\sue)\) from \cite{ball2020exact}.
 Since \(\cs(a) \sue J(F)\) iff \(a\in F\) (cf.\ \cite{ball2020exact}) it is immediate that these maps play the role of the mapping of \(\xe[L]\) and \(\ye[\Ex(L)]\) into \(\GC(L, \Ex(L), \in)\) and satisfy the universal properties of Theorem~\ref{t:polarities}.

As a consequence of this we see that, by Lemma~\ref{l:xe-pres}, \(\cs(\ARG)\) transforms exact meets into joins, i.e.\ \(\bigvee_{a\in A} \cs(a) = \cs(\bigmee A)\) for any exact meet \(\bigmee A\).

\section{The Booleanization as a polarity}

Recall from Corollary~\ref{c:all-subcoloc} that
\begin{align}
    \R(L)
    \hookrightarrow \Ex(L)
    \hookrightarrow \SE(L)
    \hookrightarrow \Fi(L)
    \label{eq:top-row-fi}
\end{align}
and also
\begin{align}
    \fit[\Sc]
    \hookrightarrow \fit[\Sb(L)]
    \hookrightarrow \So(L)
    \label{eq:top-row-su}
\end{align}
are all subcolocale inclusions. Furthermore, Lemma~\ref{l:regular} establishes that \(\R(L) = \Int(\Cl(L))\) is the Booleanization of \(\Fi(L)\).
From these fact we immediately obtain the following.
\begin{theorem}\label{t:bool}
    We have
    \[\R(L) = \Int(\Cl(L)) = \B(\Ex(L)) = \B(\SE(L)) = \B(\Fi(L))\]
    and therefore also
    \[\fit[\Sc(L)] = \B(\fit[\Sb(L)]) = \B(\So(L)).\]
\end{theorem}
\begin{proof}
    Recall that for any dense sublocale inclusion \(S \sue L\) it is the case that \(\B(S) \sue \B(L)\). If, furthermore, \(\B(L) \sue S\) then also \(\B(S) = \B(\B(S)) \sue \B(L)\). Consequently, from the subcolocale inclusions in \eqref{eq:top-row-fi} we obtain the first set of equalities since \(\R(L) = \B(\Fi(L))\), which holds by definition of \(\R(L)\).
 The second line of equalities is just a translation of the first line along the adjunction~\eqref{eq:fi-su-adj}.
\end{proof}

We also use the closure operator $\clOp(\ARG)$ on the lattice of all sublocales. For a sublocale $S$ it is defined as $\clOp(S)=\bigcap\{\cs(x):S\subseteq \cs(x)\}$, i.e.\ in our notation \(\clOp(S) = \clOp_{\cs[L]}(S)\). The following is a standard result of pointfree topology.

\begin{lemma}\label{l:cl-subloc}
    For a sublocale $S\subseteq L$ of a frame $L$, we have $\clOp(S)= \upset \bigwedge S$.
\end{lemma}

Now we rephrase Booleanization of frames in terms of polarities.
The following result highlighting a symmetry between the Booleanization of $L$ and that of $\So(L)$. In the following, $\tot=\{(x,y)\in L\times L \mid x\vee y=1\}$ and $\con=\{(x,y)\in L\times L \mid x\mee y=0\}$.

\begin{corollary}
Let $L$ be a frame. We have
\begin{enumerate}
    \item $\GC(L^{op},L,\tot)\cong\GC(\cs[L],\os[L],\sue)\cong \fit[\Sc(L)]=\B(\So(L))$.
    \item $\GC(L,L^{op},\con)\cong \GC(\os[L],\cs[L],\sue)\cong \clOp[\os[L]]=\B(\cs[L]^{op})\cong \B(L)$.
\end{enumerate}
\end{corollary}
\begin{proof}
    For the first item, the first isomorphism holds because it is a basic fact that the collection of open sublocales is isomorphic to $L$ whereas that of closed sublocales are anti-isomorphic to it and \((a,b)\in\tot\) iff \(\cs(a) \sue \os(b)\). The second isomorphism holds by Proposition~\ref{p:int-cl}. Finally, the last equality holds by Theorem~\ref{t:bool}.

    For the second item, the first two isomorphisms are shown similarly as in the first except that we needed that \((a,b)\in \con\) iff \(\os(a) \sue \cs(b)\). For the equality $\clOp[\os[L]]=\B(\cs[L]^{op})$, we note that the regular elements of $\cs[L]^{op}$, i.e. the elements of its Booleanization, are the closed sublocales of the form $\cs(x\psc)$ for some $x\in L$, as $x\mapsto \cs(x)$ determines an isomorphism $L\cong\cs[L]^{op}$. But these are precisely the closures of open sublocales as for each $x\in L$ we have $\clOp(\os(x))= \upset \bigwedge \os(x)= \upset (x\psc) = \cs(x\psc)$, where we have used Lemma~\ref{l:cl-subloc} for the first equality. The last isomorphism is once again from $L\cong \cs[L]^{op}$.
\end{proof}

\section{Closed filters, open filters, sublocales and related facts}
\label{s:closed-open}

In this final section we would like to argue that the coframe \(\Fi(L)\), importantly in the \(\sqleq\) order, in some ways resembles the coframe \(\Sl(L)\) and is an interesting discretisation of \(L\) as well.

Recall from \eqref{eq:cl-filt} that closed filters arise from closed sublocales, i.e.\ they are of the form \(\stf(\cs(a)) = \{ x \mid x \vee a = 1\}\). Similarly, we define \df{open filters} to be the principal filters since we have that \(\stf(\os(a)) = \upset a\). To make reasoning with filters more natural we write \(\cf(a)\) and \(\of(a)\) for \(\stf(\cs(a))\) and \(\stf(\os(a))\) and, furthermore, \(\Emp = L\) and \(\{1\}\) for the smallest and largest filters, respectively.

Perhaps surprisingly, we can transfer some facts about open and closed sublocales to the setting of filters. For example, recall that a frame is subfit iff its open sublocales are joins of closed sublocales (see e.g.\ Proposition V.1.4 in \cite{pp2012book}). Analogously, we obtain a new characterisation of subfitness in terms of open and closed filters.

\begin{proposition}
    \label{p:subfit-joins-closed}
    A frame \(L\) is subfit if and only if, for any \(a\in L\),
    \[\of(a) = \bigsqcup \{ \cf(x) \mid \cf(x) \sqleq \of(a) \}.\]
\end{proposition}
\begin{proof}
    First, observe that from the definitions \(\cf(x) \sqleq \of(a)\) iff \(\stf(\os(a)) \sue \stf(\cs(x))\) iff \(a \in \stf(\cs(x))\) iff \(a \vee x = 1\).
    Since the \(\sqsupseteq\) inequality in the above expression always holds, the right-hand side of the equivalence can be rewritten in the subset order as follows.
    \[
        \bigcap \{ \cf(x) \mid a \vee x = 1 \} \sue \upset a
    \]
    This inclusion can be further simplified to the first-order expression:
    \[
        \forall b\enspace
        (\forall x.\enspace a\vee x = 1 \implies b \vee x = 1)
        \implies
        a \leq b
    \]
    We see that this is precisely the axiom for subfitness.
\end{proof}

Since \(\Fi(L)\) is morally closer to \(\So(L)\) than to \(\Sl(L)\), closed filters could be rather thought of as fittings of closed sublocales. Therefore, we do not expect closed and open sublocales and closed and open behave exactly alike. Some usual properties still hold.

\begin{lemma}
    For any \(a\in L\), we have
    \,\( \of(a)\spp = \cf(a)\) \,and\, \( \of(a) \sqcup \cf(a) = \{1\}\).
\end{lemma}
\begin{proof}
    The first equality follows directly from definitions and \eqref{eq:cl-supp-open}. The second is a consequence of the first as \(\{1\} \diff \of(a) = \cf(a)\) implies \(\{1\} \sqleq \of(a) \sqcup \cf(a)\).
\end{proof}

On the other hand, some of the expected facts about open and closed sublocales do not transfer. For example, we do not in general have that supplements of closed filters are open. On the other hand, requiring this property yields a new (and perhaps surprising) characterisation of subfitness.

\begin{proposition}
    \label{p:open-closed-surprises}
    For a frame \(L\),
\begin{enumerate}
    \item \(L\) is subfit iff, for every \(a\in L\), \(\cf(a)\spp = \of(a)\); and
    \item \(L\) is a Boolean algebra iff, for every \(a\in L\), \(\of(a) \sqcap \cf(a) = \Emp\).
\end{enumerate}
\end{proposition}
\begin{proof}
    For (1), as in Proposition~2.9.2 of \cite{ball2020exact}, observe that the supplement \(F\spp\) of any filter \(F\) in \(\Fi(L)\) is the filter \(\{ b \mid \forall x\in F.\ b \vee x = 1 \}\) since \(F\spp = \{1\} \diff F\) computed in the subset order is precisely as \(\bigvee \{ \upset b \mid \upset b \cap F \sue \{1\}\}\). Furthermore, \(\cf(a)\spp = \{1\} \diff \cf(a) \sqleq \of(a)\) holds automatically since, from adjointness with \(\sqcup\), it is equivalent to \(\{1\} \sqleq \cf(a) \sqcup \of(a)\) which in turn is equivalent to \(\cf(a) = \{1\} \diff \of(a) \sqleq \cf(a)\).

    Finally, the non-automatic inequality \(\of(a) \sqleq \cf(a)\spp\) translates as
    \[ \{ b \mid \forall x.\ x \vee a =1 \implies b \vee x = 1\} = \{b \mid \forall x\in \cf(a).\ b \vee x = 1 \} \sue \upset a. \]
    Clearly, requiring this to hold for every \(a\) is precisely subfitness.

    For (2), observe that \(\of(a) \sqcap \cf(a) \sqleq \Emp\) iff \(0 \in \of(a) \cap \cf(a)\) iff there are some \(x, a'\) such that \(x\vee a = 1\) and \(a' \geq a\) and we have \(x \mee a' = 0\). Clearly, the last statement is equivalent to \(a\) being complemented.
\end{proof}

\phantomsection  %
\addcontentsline{toc}{section}{References}
\bibliographystyle{alpha}
\bibliography{refs}

\begin{thebibliography}{BMP20}

\bibitem[APP23]{arrieta2023spatial}
Igor Arrieta, Jorge Picado, and Ale\v{s} Pultr.
\newblock Notes on the spatial part of a frame.
\newblock {\em Categories and General Algebraic Structures with Applications},
  pages~--, 2023.

\bibitem[Arr21]{arrieta2021complemented}
Igor Arrieta.
\newblock On joins of complemented sublocales.
\newblock {\em Algebra universalis}, 83(1):1--11, 2021.

\bibitem[Bir48]{birkhoff48}
Garrett Birkhoff.
\newblock {\em Lattice Theory}.
\newblock American Mathematical Society, New York, third edition edition, 1948.

\bibitem[BMP20]{ball2020exact}
Richard~N. Ball, M.~Andrew Moshier, and Ale\v{s} Pultr.
\newblock Exact filters and joins of closed sublocales.
\newblock {\em Applied Categorical Structures}, 28:655--667, 2020.

\bibitem[Esc03]{escardo2003joins}
Mart{\'\i}n~H. Escard{\'o}.
\newblock Joins in the frame of nuclei.
\newblock {\em Applied Categorical Structures}, 11(2):117--124, 2003.

\bibitem[Geh06]{gehrke06}
Mai Gehrke.
\newblock Generalized kripke frames.
\newblock {\em Studia Logica}, 84(2):241--275, 2006.

\bibitem[Geh18]{gehrke2018canonical}
Mai Gehrke.
\newblock Canonical extensions: an algebraic approach to stone duality.
\newblock {\em Algebra universalis}, 79(3), 2018.

\bibitem[GF]{GFbook}
Mai Gehrke and Wesley Fussner.
\newblock Canonical extensions quickly.
\newblock (Under preparation).

\bibitem[GH01]{gehrke2001bounded}
Mai Gehrke and John Harding.
\newblock Bounded lattice expansions.
\newblock {\em Journal of Algebra}, 238(1):345--371, 2001.

\bibitem[GJ04]{gehrke2004bounded}
Mai Gehrke and Bjarni J{\'o}nsson.
\newblock Bounded distributive lattice expansions.
\newblock {\em Mathematica Scandinavica}, 94(1):13--45, 2004.

\bibitem[HM81]{hofmann2006local}
Karl~H. Hofmann and Michael~W. Mislove.
\newblock Local compactness and continuous lattices.
\newblock In {\em Continuous Lattices: Proceedings of the Conference on
  Topological and Categorical Aspects of Continuous Lattices, Bremen 1979},
  volume 871 of {\em Lecture Notes in Mathematics}, pages 209--248. Springer
  Verlag, 1981.

\bibitem[Jak20]{jakl20}
Tom{\'a}{\v{s}} Jakl.
\newblock Canonical extensions of locally compact frames.
\newblock {\em Topology and its Applications}, 273:106976, 2020.

\bibitem[Joh85]{johnstone1985vietoris}
Peter~T. Johnstone.
\newblock Vietoris locales and localic semilattices.
\newblock {\em Continuous lattices and their applications}, 101:155--180, 1985.
\newblock Pure and Applied Mathematics Vol. 101.

\bibitem[JT51]{jonssontarski1951boolean}
Bjarni J{\'o}nsson and Alfred Tarski.
\newblock Boolean algebras with operators, {I}.
\newblock {\em American Journal of Mathematics}, 73:891--939, 1951.

\bibitem[JT52]{jonssontarski1952boolean}
Bjarni J{\'o}nsson and Alfred Tarski.
\newblock Boolean algebras with operators, {II}.
\newblock {\em American Journal of Mathematics}, 74:127--162, 1952.

\bibitem[MPP22]{moshier2022some}
M.~Andrew Moshier, Jorge Picado, and Ale{\v{s}} Pultr.
\newblock Some general aspects of exactness and strong exactness of meets.
\newblock {\em Topology and its Applications}, 309:107906, 2022.

\bibitem[MPS20]{moshier2020exact}
M.~Andrew Moshier, Ale\v{s} Pultr, and Anna~Laura Suarez.
\newblock Exact and strongly exact filters.
\newblock {\em Applied Categorical Structures}, 28(6):907--920, 2020.

\bibitem[PP12]{pp2012book}
Jorge Picado and Ale\v{s} Pultr.
\newblock {\em Frames and Locales: Topology without points}, volume~28 of {\em
  Frontiers in Mathematics}.
\newblock Springer Basel, 2012.

\bibitem[PP17]{picado2017boolean}
Jorge Picado and Ale{\v{s}} Pultr.
\newblock A boolean extension of a frame and a representation of discontinuity.
\newblock {\em Quaestiones Mathematicae}, 40(8):1111--1125, 2017.

\bibitem[PP21a]{pp2021notes}
Jorge Picado and Ale{\v{s}} Pultr.
\newblock Notes on point-free topology.
\newblock {\em New Perspectives in Algebra, Topology and Categories: Summer
  School, Louvain-la-Neuve, Belgium, September 12-15, 2018 and September 11-14,
  2019}, pages 173--223, 2021.

\bibitem[PP21b]{pp2021book2}
Jorge Picado and Ale{\v{s}} Pultr.
\newblock {\em Separation in point-free topology}.
\newblock Birkh\"{a}user/Springer Cham, 2021.

\bibitem[PPT19]{picado2019Sc}
Jorge Picado, Ale\v{s} Pultr, and Anna Tozzi.
\newblock Joins of closed sublocales.
\newblock {\em Houston Journal of Mathematics}, 45:21--38, 2019.

\bibitem[Pri84]{priestley1984ordered}
H.~A. Priestley.
\newblock Ordered sets and duality for distributive lattices.
\newblock In Maurice Pouzet and Denis Richard, editors, {\em Orders:
  Description and Roles Ordres: Description et Rôles}, volume~99 of {\em
  North-Holland Mathematics Studies}, pages 39--60. North-Holland, 1984.

\bibitem[Vic97]{vickers1997constructive}
Steven Vickers.
\newblock Constructive points of powerlocales.
\newblock In {\em Mathematical Proceedings of the Cambridge Philosophical
  Society}, volume 122, pages 207--222. Cambridge University Press, 1997.

\bibitem[Wil82]{wille82}
Rudolf Wille.
\newblock Restructuring lattice theory: an approach based on hierarchies of
  concepts.
\newblock In {\em Ordered sets ({B}anff, {A}lta., 1981)}, volume~83 of {\em
  NATO Adv. Study Inst. Ser. C: Math. Phys. Sci.}, pages 445--470. Reidel,
  Dordrecht-Boston, Mass., 1982.

\end{thebibliography}

\bigskip
\paragraph*{Acknowledgement.}

Tom\'a\v s Jakl was funded by the EXPRO project 20-31529X awarded by the Czech Science Foundation and was also supported by the Institute of Mathematics, Czech Academy of Sciences (RVO 67985840).

\end{document}